 \documentclass[smallcondensed]{svjour3} 
 
  \usepackage{amsmath,amssymb}

 \usepackage{booktabs}

 \usepackage{amsfonts}
  \usepackage{pifont}
 \usepackage{graphicx}
 \usepackage{color}
 \usepackage{hyperref}
 \usepackage{geometry}
\usepackage{algpseudocode}
\usepackage{algorithm}
\algnewcommand\algorithmicinput{\textbf{Input:}}
\algnewcommand\INPUT{\item[\algorithmicinput]}

\newcommand{\phit}{\tilde{\phi}}
\newcommand{\lambdapert}{\tilde{\lambda}}
\newcommand{\Pit}{\widetilde{\Pi}}
\newcommand{\Kt}{\widetilde{K}}
\newcommand{\At}{\tilde{A}}
\newcommand{\bt}{\tilde{b}}
\newcommand{\Rt}{R^\textup{opt}}
\newcommand{\RLin}{{R}_{\hbox{\footnotesize Lin}}}
\newcommand{\RLinopt}{{R}^\textup{opt}_{\hbox{\footnotesize Lin}}}

\newcommand{\zt}{\tilde{z}}
\newcommand{\cR}{{R}}
\newcommand{\cRt}{\widetilde\cR}
\newcommand{\bbA}{K}

\newcommand{\Apert}{\widetilde{\bbA}}
\newcommand{\phipert}{\phi_{(\bbA, \Apert)}}
\newcommand{\Ft}{\widetilde{F}}
\newcommand{\ft}{\tilde{f}}
\newcommand{\htilde}{\tilde{h}}
\newcommand{\alphaup}{\alpha^\textup{up}}
\newcommand{\alphadown}{\alpha^\textup{down}}
\newcommand{\Dt}{h}

\newcommand{\order}{{\mathcal O}}
\newcommand{\rmax}{r_\textup{max}}
\newcommand{\rmin}{r_\textup{min}}

\DeclareMathOperator{\sign}{sign}

 \journalname{Journal of Scientific Computing}

\begin{document}

\title{Optimal monotonicity--preserving perturbations of a given Runge-Kutta method\thanks{The first author was supported by Ministerio de Econom{\'i}a y Competividad, Spain, Projects MTM2014-53178-P and MTM2016-77735-C3-2-P. The second and third authors were supported by KAUST Award No. FIC/2010/05-2000000231. The third author was also supported by T\'AMOP-4.2.2.A-11/1/KONV-2012-0012: Basic research for the development of hybrid and electric vehicles,  supported by the Hungarian Government and co-financed by the European Social Fund.}}
\author{Inmaculada Higueras  \and David I. Ketcheson	\and Tiham\'er A. Kocsis}

\institute{
I. Higueras \at Public University of Navarre, Pamplona 31006, Spain. \email{higueras@unavarra.es}  
 \and D.I. Ketcheson \at King Abdullah University of Science and Technology (KAUST), Thuwal 23955-6900, Saudi Arabia. \email{david.ketcheson@kaust.edu.sa}
 \and T.A. Kocsis \at Sz\'echenyi Istv\'an University, Gy\H{o}r, H-9026, Hungary.  \email{katihi@sze.hu}}

\date{\today}

\maketitle

\begin{abstract}
Perturbed Runge--Kutta methods (also referred to as downwind Runge--Kutta methods)
can guarantee monotonicity preservation under larger step sizes relative to their traditional
Runge--Kutta counterparts.  In this paper we study the question of how
to optimally perturb a given method in order to increase the radius of absolute monotonicity (a.m.).
We prove that for methods with zero radius of a.m., it is always possible to give
a perturbation with positive radius. We first study methods for linear problems and
then methods for nonlinear problems.  In each case, we prove upper bounds on the radius 
of a.m., and provide algorithms to compute optimal perturbations.
We also provide optimal perturbations for many known methods. 
\end{abstract}

\keywords{
 Strong Stability Preserving, Monotonicity, Runge-Kutta methods, time discretization 
} 
 
\subclass{65L06, 65L20, 65M20}

\section{Introduction}
In this work we are concerned with the numerical solution
of initial value ordinary differential equations:
\begin{align}
u'(t) & = f(u (t))\, , & u(0) = u_0\, . \label{ode}
\end{align}
In many physical problems $f$ is dissipative, i.e.
the exact solution satisfies
\begin{align} \label{dissipative}
\frac{d}{dt}\|u (t)\| \le 0,
\end{align}
where $\|\cdot\|$ denotes a convex functional (e.g., a norm or semi--norm). 
A sufficient condition for \eqref{dissipative} is that $f$ be monotone under an
explicit Euler step:
\begin{align} \label{FEcond1}
	\| v + h f(v)\| & \le \|v\|, & \mbox{for all $v$, and for $h$ satisfying } 0 \le h \le h_0,
\end{align}
where $h_0>0$ (in general $h_0$ may depend on $v$).  We refer to
\cite[p.~1-2]{higueras2005a}  and \cite[p.~501]{Kraaijevanger1991} for details.\footnote{
Although the results in  \cite{Kraaijevanger1991} are given in the
context of contractivity, they are also relevant to the preservation of monotonicity. In
\cite[Thm. 5.1]{Kraaijevanger1991}, quotients $m_\tau[x,y]$ and one-sided
Gateaux variations ($m_+[ x,y]$, $m_-[ x,y]$),  are used for $x=u-\tilde u$ and
$y=f(u)-f(\tilde u)$ to obtain the contractivity property $\|u(t)-\tilde
u(t)\|\leq \|u(t_0)-\tilde u(t_0)\|$ for $t\geq t_0$. In the context of monotonicity,
we simply take $x=u$ and $y=f(u)$ to obtain the monotonicity
property $\|u(t)\|\leq \|u(t_0)\|$ for $t\geq t_0$.}

Let $u_{n}, u_{n+1}$ denote approximations, computed by some numerical integrator,
to the solution at successive time steps $t_{n}$ and $t_{n+1}=t_{n}+h$.
Under the forward Euler monotonicity condition \eqref{FEcond1}, 
it is possible to prove 
that many Runge--Kutta and linear multistep methods also give monotone
solutions; i.e., solutions that satisfy
\begin{align} \label{monot}
	\|u_{n+1}\| & \le \|u_{n}\|, & \mbox{for $h$ satisfying } 0 \le h \le R \, h_0.
\end{align}
Such methods are known as  strong stability preserving (SSP) methods, and
the factor $R$ is known as the radius of absolute monotonicity or SSP coefficient
of the method.
SSP methods necessarily have non-negative coefficients, since the monotonicity
property is proved using \eqref{FEcond1} and convexity.
Results on numerical preservation of some other properties,
like non-negativity \cite{Horvath2005} or discrete maximum-principle \cite{zhang2010}, can also be obtained in the SSP framework.

Monotonicity cannot be ensured using only assumption \eqref{FEcond1} for
methods with negative coefficients \cite[Thm. 4.2]{Kraaijevanger1991}, or even
for some methods (such as the classical fourth-order Runge--Kutta method) with
non-negative coefficients \cite[Thm. 9.6]{Kraaijevanger1991}.
However, in some problems (such as those of Section \ref{sec:example} below)
it happens that $f$ satisfies property \eqref{FEcond1} also for negative
step sizes.  In other cases, $f$ is a dissipative approximation of a 
conservative operator, in which case one may devise a second approximation
$\ft$ that is dissipative for negative step sizes; i.e.
\begin{align} \label{FEcond2}
	\| v - h \tilde f(v)\| & \le \|v\|, & \mbox{for all $v$, and for $h$ satisfying } 
    										0 \le h \le \tilde{h}_0, 
\end{align}
where $\tilde{h}_0>0$. This situation arises naturally in the context of hyperbolic PDE
semi-discretizations, where $f$ is upwind-biased and $\ft$ is downwind-biased;
typically $\tilde{h}_0 = h_0$. 

The function $\ft$ is to be used in place of $f$ wherever 
a negative coefficient  appears in the time integration method, in order to ensure 
monotonicity of the overall method.  Introduction of $\ft$
makes it possible to ensure monotonicity for a broader class of methods,
including the classical Runge--Kutta method of order four.  It also makes it possible
to ensure monotonicity for many methods under larger step sizes.

During the last quarter century, a number of additional authors have studied 
monotonicity for methods that use $\ft$ (see e.g., 
\cite{shu1988,gottlieb1998total,ruuth2004,higueras2005a,higueras2006,ruuth2006,gottlieb2006,DoHiMa,Higueras2010,ketcheson2011a}).
The main motivation for this work has been to
break the  ``order barrier'' that restricts explicit Runge--Kutta methods to order
four and to find new methods with larger SSP coefficient \cite{shu1988,gottlieb1998total,ruuth2004,higueras2005a,higueras2006,ruuth2006,gottlieb2006,ketcheson2011a}),  or  to explain why some non-SSP
methods preserve strong stability properties like non-negativity and a discrete
maximum principle \cite{DoHiMa,Higueras2010}.
In this context, numerical optimization of the SSP coefficient for Runge--Kutta methods with negative
coefficients was conducted for explicit methods in \cite{ruuth2004,ruuth2006,gottlieb2006} 
and for implicit methods in \cite{ketcheson2011a}.  
In each case, optimization was
carried out over methods with a specified order and number of stages.

 Methods that use both $f$ and $\ft$ can naturally be viewed as
{\em perturbed Runge--Kutta methods}. Although they are also connected to
additive Runge--Kutta methods (see \cite{higueras2005a,higueras2006}),
in the present work we will employ the perturbation viewpoint,
and refer to methods that use downwind discretization as perturbed Runge--Kutta methods.

\subsection{Perturbed Runge--Kutta methods}
A Runge-Kutta method applied to the initial value problem \eqref{ode}
computes approximations $u_n \approx u(t_n)$ by
\begin{subequations} \label{Runge--Kutta}
\begin{align}
    Y & = u_n e + h \bbA F\, ,  \\
    u_{n+1} & = Y_{s+1}.
\end{align}
\end{subequations}
Here $s$ is the number of stages, $e$ is a vector whose entries are equal to one, $Y$ is the vector containing the stage values and the numerical solution, $Y=(Y_1, \ldots, Y_s, Y_{s+1})^t$,  
$[F]_i = f(Y_i)$, and $K$ is the $(s+1)\times(s+1)$ matrix of Butcher
coefficients:
$$\bbA=\begin{pmatrix} A & 0 \\ b^t & 0 \end{pmatrix}\, . 
$$
In this work we study perturbations of a Runge--Kutta method $\bbA$ to solve problem \eqref{ode}. 
To define a perturbed method, we introduce a second coefficient matrix
$$\Apert=\begin{pmatrix} \tilde A & 0 \\ \tilde b^t & 0 \end{pmatrix}\, ,
$$
where the matrix $\tilde A$ has the same structure
(strictly lower-triangular, lower-triangular, or full)
as the matrix $A$.  We also introduce
a function $\ft$ such that $\ft\approx f$.
We assume that $f$ and $\ft$ satisfy the explicit Euler assumptions 
\eqref{FEcond1} and \eqref{FEcond2}, respectively, with $\htilde_0 = h_0$.
\begin{definition} \label{def:pertRK}
A {\it perturbed Runge--Kutta method} $(\bbA, \Apert)$ takes the form
\begin{subequations} \label{PRunge--Kutta}
\begin{align}
    Y & = u_n e + h \bbA F + h\Apert(F-\Ft)\, ,  \\
    u_{n+1} & = Y_{s+1},
\end{align}
\end{subequations}
where $[\Ft]_i=\tilde f(Y_i)$.  
\end{definition}
As far as we know, the first attempt to perturb a given Runge--Kutta
method to obtain non-trivial SSP coefficient was made in \cite{shu1988},
where the classical fourth--order  Runge--Kutta scheme is perturbed to obtain a
non-trivial SSP coefficient. The goals of this paper are to perform a rigorous
study of perturbed Runge--Kutta methods and propose algorithms to obtain
perturbations of a given Runge-Kutta method with optimal SSP coefficient.

 \begin{remark} {\it (Perturbed methods as additive schemes)}\label{rem:additive}
Observe that method \eqref{PRunge--Kutta}
may be viewed as approximating the solution of the perturbed problem
$$u'(t) = f(u) + (f(u) - \ft(u)),$$
where $\ft\approx f$,  with the additive method $(\bbA, \Apert)$.\hfill $\square$
 \end{remark}
 
\subsubsection{Relation between the unperturbed method and the perturbed method}
The present work is based on the premise that the behavior of a perturbed method
$(\bbA, \Apert)$ is related to the properties of the unperturbed method with coefficients
$\bbA$.  To see why this is the case, observe first that the perturbed method
\eqref{PRunge--Kutta} reduces to the Runge--Kutta method 
\eqref{Runge--Kutta} when  $\ft = f$.  Furthermore, as $\ft \to f$, the
perturbed method solution (given by \eqref{PRunge--Kutta}) obviously tends to the
unperturbed method solution (given by \eqref{Runge--Kutta}). 
In practice, for hyperbolic
problems, $\ft$ and $f$ are discretizations of a spatial differential operator \cite[p.144]{shu1988}, and
their difference can be made arbitrarily small by increasing the accuracy of these
discretizations.   As far as convergence is concerned, the reasoning in \cite[p. 933-934]{higueras2005a} shows that, for stable Runge--Kutta methods,  the perturbed method retains the order of the unperturbed one, provided that  $f-\ft$ is small enough. 
 Herein we are particularly interested in high-order time discretizations,
intended to be paired with high-order spatial discretizations, for which the difference
$f-\ft$ is very small. 

Given the close relationship between the perturbed method and its
unperturbed counterpart, it makes sense to consider developing perturbed
versions of existing methods, in order to take advantage of the substantial
amount of work that has gone into designing those methods. 

\subsection{Two motivating examples}\label{sec:example}
To demonstrate the usefulness of the present work, we consider two
numerical experiments.   In both, the Runge--Kutta methods are used
in the standard way, and an equivalent reformulation allows us to analyze their behavior using the
formalism of  perturbed  schemes with $\ft=f$; in other words, 
 the Runge--Kutta methods are perturbed fictitiously.

The first example shows how this work can better be used to
{\em understand the behavior of standard (unperturbed) Runge--Kutta methods}.
The second one shows that  ``perturbing" a  robust  Runge--Kutta method  with many important features  can be advantageous
versus using an optimized SSP perturbed method.

\subsubsection{Example 1}
We integrate the initial value problem
\begin{equation}
u'(t) = \sign(\sin(t))u(t)(1-u(t))\label{ExamPos}
\end{equation}
on the interval $t\in[0,100]$ with initial condition $u(0)\in(0,1)$.
The true solution remains in the interval $(0,1)$, and the explicit Euler method
keeps the solution in this interval if the step size satisfies
$-1 \le h < 1$  (note that negative step sizes are included here).  
We will apply some well-known Runge--Kutta methods to this problem
and consider two initial values:
$u(0) = 10^{-8}$ and $u(0)=1-10^{-8}$ (these values are
chosen because initial values very close to zero or unity are the
most challenging; testing other initial values in $[0,1]$ does not seem to
change the results found below).

We first consider the explicit midpoint method:
\begin{align*}
    y_1 & = u_n\,,  \\
    y_2 & = u_n+ \frac{h}{2} f(y_1)\, ,  \\
    u_{n+1} & = u_n + h f(y_2).
\end{align*}
For this method, the formula for $y_2$ is an Euler step, but the formula for
$u_{n+1}$ cannot be written
as a convex combination of forward Euler steps, so the standard theory
of strong stability preservation does not guarantee invariance of the interval
$(0,1)$ under any step size.  Nevertheless, experimentally we observe that
the interval is preserved for step sizes up to $h\approx 0.73$ (see Table  \ref{example1}).

The theory in the present paper explains this result rather precisely.
Since $f$ satisfies both \eqref{FEcond1} and \eqref{FEcond2}  for $h_0=\tilde h_0=1$, we can
formally introduce a function $\ft=f$ to facilitate the analysis.
We thus ``perturb" the midpoint method, replacing the formula for $u_{n+1}$ with  the equivalent expressions
\begin{align}
    u_{n+1} & = u_n  + h f(y_2) +  h \, \frac{r}{2} \, \left(f(y_1) - \ft(y_1)\right) \label{midpoint_int1} \\
            & =   r \left(y_2 + \frac{h}{r}f(y_2)\right)+ (1-r)\left(y_1 - \frac{h}{r}\ft(y_1) \right) \label{midpoint_int2}
\end{align}
where $r = \sqrt{3}-1$.  The last formula above shows that $u_{n+1}$ can be
written as a convex combination of forward Euler steps with step size $h/r$,
one using $f$ and one using $\ft$.  Of course, since $\ft=f$ this is in fact
the same midpoint method, but writing it this way allows us to prove that it
preserves the interval $(0,1)$ for step sizes up to $\sqrt{3}-1 \approx 0.73$. The 
perturbed forms \eqref{midpoint_int1} and \eqref{midpoint_int2} correspond to
expressions \eqref{PRunge--Kutta} and \eqref{dwRunge--Kutta_cso}, respectively,
for the explicit midpoint method (see \eqref{midpointRKpert}). 

Results for some additional methods are given in Table \ref{example1}.
The value $R(K)$ is the SSP coefficient, which is also the theoretical maximum step size
for preservation of the interval $(0,1)$ that can be guaranteed based
on considering only condition \eqref{FEcond1}.  The value $h_\text{obs}$
is the largest step size (truncated to 2 decimal places) observed to preserve
the invariant interval
in practice.  Finally, the value $\Rt(K)$ gives the step size that
can be guaranteed to preserve the interval using the tools developed in
the present work, by finding an optimal perturbation.  The values
$\Rt(K)$ do a much better job of predicting (or explaining) the
behavior of the methods for this problem.

\begin{table}
\begin{center}
\begin{tabular}{l|ll|l}
\toprule
Method & $R(K)$ & $h_\text{obs}$ & $\Rt(K)$ \\
\midrule
Forward Euler & 1 &      1.00 &          1 \\
Midpoint RK2  & 0 &      0.73 &      0.732 \\
Heun33 \cite{heun1900neue} & 0 & 0.91 &      0.776 \\
RK4 (Kutta)          & 0 & 1.24 &      0.685 \\
Merson \cite{merson1957operational} & 0 & 0.29 &      0.242 \\
\bottomrule
\end{tabular}
\caption{Theoretical  ($R(K)$)  and observed  ($h_\text{obs}$)  step sizes for preserving the invariant
interval $(0,1)$ for problem \eqref{ExamPos}.  Values in the last column  ($\Rt(K)$)  are obtained using the tools presented herein.\label{example1}}
\end{center}
\end{table}

\subsubsection{Example 2}
 We integrate the problem 
\begin{align} \label{lyproblem}
    u'(t) = 5 \, u\, (1-u)\left(u-\frac{1}{2}\right)
\end{align}
with initial condition $u(0)=0.49$ up to time $T=10$.  This problem was
proposed in the context of hyperbolic PDEs in \cite{levequeyee1990} and
has been studied extensively.  The values $u=0$ and $u=1$ are stable equilibria
while the value $u=1/2$ is unstable; the exact solution remains in the interval
$[0,1]$.  It can be shown that the forward Euler method
preserves the interval $u\in[0,1]$ for step sizes in the approximate range
$-16/5 \le \Delta t \le 2/5$   (note that negative step sizes are included here)  \cite[Lemma 8.1]{DoHiMa}. 

 We consider two methods:
the 5th order Bogacki-Shampine method \cite{Bogacki1996} (BS75)  and the optimized 5th order 7-stage
SSP perturbed method  (SSP75) \cite{ruuth2004}.   Method BS75 has negative coefficients,
so traditional SSP theory does not guarantee invariance of $[0,1]$ under any non-zero step size; using
the approach in this paper it can be proven to preserve this interval for step sizes up
to approximately $0.125$.  The SSP75 method can be proven to preserve this interval
for step sizes up to approximately $0.558$.  The  values  $0.125$ and $0.558$ are obtained from  \eqref{eq:stepsize_r}  for $h_0=2/5$ and the data in Table \ref{tbl:explicit-methods1}.

In Figure \ref{comparison} we plot the error versus the time step for each method.
Open circles indicate solutions that contain values outside the interval $[0,1]$.
Notice that the BS75 method performs better both in terms of accuracy and strong
stability preservation. 

\begin{figure}\begin{center}\includegraphics[scale=0.6]{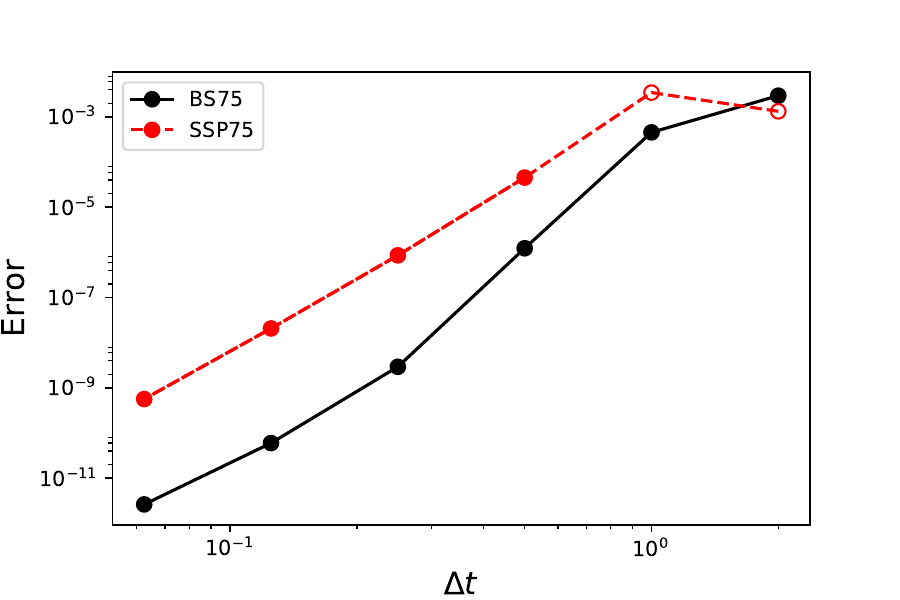}
\caption{Accuracy of BS75 and SSP75 methods for \eqref{lyproblem}.  Open circles
indicate the presence of negative solution values.}
\label{comparison}
\end{center}
\end{figure}

Recall that the scheme SSP75 has been optimized to achieve the largest
SSP coefficient but many other relevant properties are not taken into account.
However, the Bogacki-Shampine method has been carefully constructed to optimize
several properties, including accuracy.
The difference in the error constants for Bogacki-Shampine and SSP75 methods,
approximately $2.2\times 10^{-5}$ and  $2.7\times 10^{-3}$, respectively,
explains the observed accuracies.  Clearly, a method optimized for properties
other than the SSP coefficient may be useful even for problems where strong
stability properties are paramount. 

 Production implementations of modern IVP solvers include many
 important features, such as continuous output, error estimation, and automatic
 step size control \cite{HaWa}. The BS75 method, for instance, includes all of these
 features.  None of these have been developed for existing high-order
 optimal downwind SSP methods, so such methods may not be a reasonable option
 when an efficient and robust solution is required.  By instead using a
 perturbation of an existing method, all of these features can be used in the
 usual way.

\subsection{Scope and outline}\label{questions}
In the present work we seek to answer the following questions:
\begin{enumerate}
        \item Can every Runge--Kutta method be perturbed in a way that yields
                a positive radius of absolute monotonicity? \label{Q1}
        \item What {\em a priori} limits are there on the radius of absolute monotonicity
                obtained by perturbing a given method? \label{Q2}
        \item Can optimal SSP methods from the literature be perturbed in order
                achieve an even larger radius of absolute monotonicity? \label{Q3}
        \item Given a fixed Runge--Kutta method, what perturbation results in the largest
                radius of absolute monotonicity for the perturbed method? \label{Q4}
        \item How can that perturbed method be found? \label{Q5}
        \item To begin with, what are the answers to the questions above if
                only linear problems are considered? \label{Q6}
\end{enumerate}
 
In Theorem \ref{thm:feasible}, we prove that the answer to question \ref{Q1} is affirmative.
Theorems \ref{prop_vpos} and \ref{Pbound} answer question 2 by providing
simple upper bounds; Theorem \ref{Pbound} implies that the answer to question \ref{Q3}
is negative for most known optimal SSP methods.
In Section \ref{sec:algorithms} we answer questions \ref{Q4} and \ref{Q5} by
giving two algorithms for computing optimal perturbations.  The first is
provably correct but approximate, while the second is heuristic but exact and
agrees with the first in all cases we have tested.  Both are applicable only to
explicit methods.  These algorithms have been implemented in the free
open-source software package Nodepy \cite{nodepy} and can easily be applied
to any desired method.
We conclude Section \ref{sec:nonlinear} with an application of the theorems and
algorithms to optimal perturbations of some Runge--Kutta methods from the literature  and a numerical test.
Among the results is the first truly optimal perturbation for the classical
4th-order method of Kutta.

We deal with application to linear problems first, in Section \ref{sec:lin:prob}.
For explicit methods applied to linear problems, the questions above can be
cast in terms of absolute monotonicity of the (bivariate) stability polynomial.
In Section \ref{sec:upperbound:lin} we prove a general upper bound on the radius
of absolute monotonicity of the stability polynomial of an explicit perturbed Runge--Kutta method
with $s$ stages and linear order $p$.
 In Section \ref{sec:upperlin:num}, we provide an algorithm for
computing tighter bounds, and tabulate some of the resulting numerical values. 
Examples of optimal methods are given in Section \ref{sec:lin-examples}.

Section \ref{sec:conclusions} contains some conclusions as well as some open
questions to be studied in the future.  

In Section \ref{sec:proofs} we give  the  proofs of the results in the paper together with an auxiliary lemma. We  have collected them in a separate section in order to not interrupt the reading of the paper. 

Finally, in the Appendix we give some
details on perturbations for the family of second order 2-stage methods and the
classical fourth order Runge--Kutta method.  

Computer code to reproduce some of the examples in this paper, including the two
examples above, can be found at  \cite{HigKetKocRepository}. 

\section{Explicit perturbed Runge--Kutta methods for linear problems}\label{sec:lin:prob}

To study the behavior of the perturbed Runge--Kutta method $(\bbA,\Apert)$ 
for linear problems, we apply it to a linear scalar test problem,
setting $f(u)=\lambda u$ and $\ft(u)=\lambdapert u$ in \eqref{PRunge--Kutta}.
This results in the iteration $$u_{n+1} = \phipert(z, -\zt) \, u_n,$$
where $z=h\lambda$, $\zt=h\lambdapert$ and
\begin{align}
\phipert(z, \zt)= 1 +  \left(z b^t + (z+\zt) \bt^t\right)\left(I -  z A - (z+\zt) \At \right)^{-1} e\, . \label{bistabfunc}
\end{align}
\begin{definition}
We refer to \eqref{bistabfunc} as the {\it stability function of the perturbed
Runge--Kutta method} $(\bbA,\Apert)$.   
\end{definition}

Following the steps in \cite[Prop. 3.2]{HaWa},  function \eqref{bistabfunc} can also be written as
\begin{align}
 \phipert(z, \zt)= \frac{\hbox{det} \left(I- z (A -e b^t) - (z+\zt) (\At -e \bt^t) \right)}{\hbox{det}  \left(I - z A - (z+\zt) \At\right)}\, .   \label{bistabfuncDet}
\end{align}

The stability function $\phi$ in \eqref{bistabfunc} is a rational function $\psi=P/Q$, where $P$ and $Q$ are polynomials in the complex variables $z$ and $\tilde  z$, both with real coefficients. A function $\psi$ of this type is said to be absolutely monotonic (a.m.) at a given point 
 $(\xi, \tilde \xi)\in \mathbb{R}^2$ if $Q(\xi, \tilde \xi)\neq 0$ and  $(d^{j+k}\psi/dz^jd\tilde z^k) (\xi , \tilde \xi)\geq 0$, $j = 0, 1, \ldots$, $k = 0, 1, \ldots$  (see, e.g., \cite[Def. 2.7]{higueras2006}).  
 
This definition is an extension of the one given in \cite[Def. 2.1]{Kraaijevanger1991} for the unidimensional case: given a rational function $\psi=P/Q$, where $P$ and $Q$ are polynomials in the complex variable $z$, both with real coefficients, we say that $\psi$ is absolutely monotonic (a.m.) at a given point 
  $\xi\in \mathbb{R}$ if $Q(\xi)\neq 0$ and all the derivatives $(d^k\psi/d z^k)(\xi)\geq 0$, $k=0, 1, 2, \ldots$

\begin{definition}
Given a function $\psi(z, \zt)$, we define the {\it radius of absolute monotonicity}  as 
\begin{align}
R(\psi)=  \sup\left\{ r\in \mathbb{R} \, | \,  r=0, \hbox{ or }  r>0\, , \, \hbox{and }
\psi(z, \zt) \text{ is a.m. at } (-r, -r) \right\}\, . \label{threshold:pert}
\end{align}
\end{definition}
Observe that, if $\psi(z, \zt)$ is a bivariate polynomial of combined degree $s$, for $r \le R(\psi)$  we can write 
\begin{align} \label{dwtaylor1}
\psi(z, \zt) = \sum_{j=0}^s \sum_{\ell=0}^j \gamma_{j\ell} 
  \left(1+\frac{z}{r}\right)^{j-\ell} \left(1+\frac{\zt}{r}\right)^{\ell}\, , 
  \quad \quad \mbox{with } \gamma_{j\ell} = \frac{r^j}{j!}\frac{\partial^j\psi}{\partial z^{j-\ell} \partial \zt^{\ell}}(-r, r),
\end{align}
where the coefficients $\gamma_{j\ell}$ are non-negative. 

 \begin{definition} Given a perturbed Runge--Kutta method \eqref{PRunge--Kutta} with coefficients $(\bbA, \Apert)$,
 we define the {\it threshold factor}  $\RLin(\bbA, \Apert)$ as the radius of absolute monotonicity of its stability
 function:
\begin{align}
 \RLin(\bbA, \Apert)=R(\phipert)\, .  \label{eq:thersfact}
 \end{align}
 \end{definition}
 
 The quantity $\RLin(\bbA,\Apert)$ is referred to as the {\it threshold factor} 
due to its role in the step size for monotonicity.  
The following
theorem appeared previously as \cite[Thm. 4.6.2]{ketchesonphdthesis}. Its proof, given in Section \ref{sec:proofs}, is based on the fact  that, for explicit schemes,  the stability function $\phipert(z,\zt)$ is a bivariate polynomial of combined degree $s$ and thus, for $r \le \RLin(\bbA,\Apert)$,  
it can be written in the form \eqref{dwtaylor1}
with non-negative coefficients $\gamma_{j\ell}$.
\begin{theorem}\label{Theorem:1}
Let a consistent perturbed $s$-stage explicit Runge--Kutta method $(\bbA, \Apert)$
be given with stability function $\phipert$,
and let $\|\cdot\|$ be a convex functional.
Consider the numerical solution
\begin{equation}
u_{n+1}=\phipert(h L, -h \tilde{L}) \, u_n\, , \label{eq:linODE}
\end{equation}
where $L$ and $\tilde L$ are linear operators such that $L \tilde{L}= \tilde{L} L$ and 
\begin{align*}
\|I + hL\|  \le 1\, ,   \quad 		   
\|I - h\tilde{L}\|   \le 1\, , \qquad    0   \le h \le h_0\, .
\end{align*}
Then the numerical solution \eqref{eq:linODE} satisfies the monotonicity condition \eqref{monot} for
step sizes
$$0 \le h \le \RLin(\bbA,\Apert) \, h_0.$$
\end{theorem}

Consequently, the larger $\RLin(\bbA,\Apert)$ is, the larger is the step size restriction for montonicity.
For a given Runge--Kutta method \eqref{Runge--Kutta} with coefficients $\bbA$,
we are interested in determining perturbations $\Apert$ that give the largest 
threshold factor. 

\begin{definition}
The {\it  threshold factor of the optimal perturbation} is
given by
\begin{align} \label{opt-thresh-fact}
  \RLinopt(\bbA)=  \sup_{\Apert} \RLin(\bbA, \Apert)\, ,
\end{align}
where, in order to preserve the explicit nature of the method, the supremum in \eqref{opt-thresh-fact} is taken over all strictly lower
triangular matrices $\Kt$.
A perturbation ${\Apert}$ such that 
$$\RLin(\bbA, \Apert) = \RLinopt(\bbA)\, ,  $$
will be called an {\it optimal perturbation of the method $\bbA$ for the linear
problem}.  
\end{definition}

Taking $\Apert=0$ gives a (not perturbed) Runge--Kutta method \eqref{Runge--Kutta} and a (not perturbed) stability function $\phi_\bbA$.
In this case we denote the threshold factor $\RLin(\bbA, 0)$ simply by $R(\phi_\bbA)$.
Clearly
\begin{align}R(\phi_{\bbA}) \le \RLinopt(\bbA)\, .\label{bound:RRL}
\end{align}
In the next section, we give upper bounds on $\RLinopt(\bbA)$.

\subsection{Upper bounds on the threshold factor for optimal perturbations} \label{sect:opt-dw-tf}
In this section we consider the set $\Pit_{s,p}$, with $p\le s$, defined as follows. 
\begin{definition}
We define   $\Pit_{s,p}$, with $p\le s$, as the set of bivariate polynomials with
the following properties:
\begin{enumerate}
    \item $ \psi(z,\zt) = \displaystyle \sum_{j=0}^p \frac{z^j}{j!} 
            + \sum_{j=p+1}^s \sigma_j z^j + (z+\zt)\, \Psi(z,\zt) $;
    \item $\Psi$ is a polynomial of combined degree at most $s-1$.  
\end{enumerate}
\end{definition}
Observe that if $\psi(z,\zt)\in \Pit_{s,p}$, then 
\begin{align}
\psi(z,-z)= \exp(z)+ \order(z^{p+1})\, .\label{eq:approx_exp}
\end{align}
The following result  explains the interest in studying the set $\Pit_{s,p}$.
\begin{proposition}\label{Proposition:1}
Let $\bbA$ be an explicit  $s$-stage Runge--Kutta method with linear order $p$. If $\phi_{(\bbA,\Apert)}$ is the stability function of the perturbed   Runge--Kutta method $(\bbA,\Apert)$, then 
 $\phi_{(\bbA,\Apert)} \in \Pit_{s,p}$.
\end{proposition}

The aim of this section is to investigate
\begin{align}\cRt_{s,p}=\sup\left\{R(\psi) \, |\, \psi(z, \zt)\in\Pit_{s,p}\right\}\, .\label{def:rsp}
\end{align}
Clearly,  by Proposition \ref{Proposition:1}, $\cRt_{s,p}$ is an upper bound of the threshold factor of the optimal perturbation defined by \eqref{opt-thresh-fact} (see too \eqref{eq:thersfact}),
\begin{align}\RLinopt(\bbA)\leq \cRt_{s,p}\, . \label{bound:RLRsp}
\end{align}

\begin{remark} \label{rem:realization} {\it (Realizable polynomials)}
We remark that not all polynomials in $\Pit_{s,p}$ can
be realized as the stability function of an $s$-stage perturbed Runge-Kutta
method \eqref{PRunge--Kutta}.
Thus, inequality \eqref{bound:RLRsp} is often strict (see Example \ref{ex:2eo2} below). In case the optimal polynomial is
realizable, the corresponding method may be
of interest for the integration of linear systems.\hfill $\square$
\end{remark}

The rest of the section is organized as follows. In Subsection \ref{sec:upperbound:lin} we give an upper bound for
$\cRt_{s,p}$.  In Subsection \ref{sec:upperlin:num}, we give an algorithm to
compute, $\cRt_{s,p}$ for given $s$ and $p$, along with numerical values.

\subsubsection{Upper bound on ${\cRt_{s,p}}$}
\label{sec:upperbound:lin}

The following upper bound on $\cRt_{s,p}$ is proved in Section \ref{sec:proofs}.

\begin{theorem}\label{Theorem:2} The coefficient $\cRt_{s,p}$ defined by \eqref{def:rsp} has the following upper bound
\begin{align}\cRt_{s,p}\leq \sqrt[p]{s(s-1)\cdots (s-p+1)}\, .  \label{incRtsp}
\end{align}
\end{theorem}

Consequently, from \eqref{bound:RLRsp}, we obtain the following bound for the threshold factor of the optimal perturbation 
\begin{align*}\RLinopt(\bbA)\leq \sqrt[p]{s(s-1)\cdots (s-p+1)}\, .  \label{upper:RLinopt}
\end{align*}

\subsubsection{Numerical computation of  $\cRt_{s,p}$}\label{sec:upperlin:num}
In this section we provide a means to compute tighter values of $\cRt_{s,p}$ using linear programming.
The material in this section closely follows \cite[Sect. 4.6.2]{ketchesonphdthesis}.

In order to obtain these bounds, for each $(s,p)$ we are going to construct functions $\psi(z,\zt)\in \Pit_{s,p}$ that  can be written in the form \eqref{dwtaylor1} for some $r>0$ with non-negative coefficients    $\gamma_{j\ell}$. Observe that these polynomials can be constructed if $\gamma_{j\ell}$ and $r$ are given.  
From \eqref{dwtaylor1}, after considerable manipulation we find that 
$\psi(z,-z)=\sum_{i=0}^s C_i z^i$ where
\begin{align*}
C_i(r,\gamma) = \sum_{j=i}^s \sum_{\ell=0}^j \gamma_{j\ell} \sum_{m=\max(0,i-\ell)}^{\min(i,j-\ell)} 
{j-\ell \choose m}{\ell \choose i-m} \frac{(-1)^{i-m}}{r^i}\, , 
\end{align*}
where $\gamma$ is a vector whose components are the coefficients $\gamma_{j,\ell}$.

Hence we have the following problem for existence of a polynomial
\eqref{dwtaylor1} with perturbed threshold factor at least $r$ and order at least $p$:
\begin{subequations}
\label{optprob3}
\begin{align}
\mbox{Given } r>0, \mbox{ find $\gamma$ such that} & \nonumber \\
& \gamma_{j\ell} \ge 0 & 0 \le \ell\le j \le s \\
& C_i(r,\gamma) = \frac{1}{i!} & 0 \le i \le p.
\label{eqconstr}
\end{align}
\end{subequations}
Since \eqref{eqconstr} is a system of linear equations (in $\gamma$)
then for any given value of $r$ \eqref{optprob3} represents a linear programming
feasibility problem.  Hence we can use bisection and an LP solver to find 
the largest value of $r$ satisfying \eqref{optprob3},  as was done for similar problems in \cite{ketcheson2008,ketcheson2009a}.  
Table \ref{dwtable} gives the computed values of $\cRt_{s,p}$ for $s$ and $p$ up to ten.

\begin{table}
\begin{center}
\begin{tabular}{|l|rrrrrrrrrr|}
\hline
s p  & 1 & 2 & 3 & 4 & 5 & 6 & 7 & 8 & 9 & 10 \\ \hline
1  & 1.00  & &  &  &  &  &  &  &  &  \\
2  & 2.00  & 1.41 &  &  &  &  &  &  &  &  \\
3  & 3.00  & 2.45 & 1.60 &  &  &  &  &  &  &  \\
4  & 4.00  & 3.46 & 2.49 & 2.00 &  &  &  &  &  &  \\
5  & 5.00  & 4.47 & 3.20 & 2.94 & 2.18 &  &  &  &  &  \\
6  &  6.00  & 5.48 & 4.00 & 3.65 & 3.11 & 2.58 &  &  &  &  \\
7  & 7.00  & 6.48 & 4.86 & 4.45 & 3.88 & 3.55 & 2.76 &  &  &  \\
8  & 8.00  & 7.48 & 5.77 & 5.31 & 4.57 & 4.32 & 3.72 & 3.15 &  &  \\
9  & 9.00  &  8.49 & 6.62 & 6.22 & 5.24 & 5.02 & 4.52 & 4.14 & 3.33 &  \\
10 & 10.00 & 9.49 & 7.42 & 7.09 & 5.95 & 5.70 & 5.25 & 4.96 & 4.32 & 3.73 \\ 
\hline
\end{tabular} 
\caption{$\cRt_{s,p}$: upper bounds on $\RLinopt(\bbA)$, the threshold factors for optimal perturbations\label{dwtable}}
\end{center}
\end{table}

\subsection{Examples}\label{sec:lin-examples}  
In Section  \ref{sec:polynomials} we give some examples of  polynomials achieving ${\cRt_{s,p}}$; in Section \ref{sec:ex_optpert} we study optimal threshold factors for perturbations $\RLinopt(\bbA)$ of specified  Runge--Kutta methods $\bbA$.

\subsubsection{Polynomials achieving ${\cRt_{s,p}}$} \label{sec:polynomials}
The algorithm just described also provides coefficients for an optimal
polynomial $\psi_{s,p} (z, \zt)$, which may or may not be realizable as the stability function of
a perturbed Runge--Kutta method. Observe that all of them belong to $\Pit_{s,p}$ and thus  $\psi_{s,p} (z, -z)$ is an order $p$ approximation of $\exp(z)$ (see \eqref{eq:approx_exp}). 

By computing optimal polynomials with $p=1$ and $p=2$ we
arrived at the following results.

\begin{proposition}\label{prop:rs1}
For $p=1$ we have $\cRt_{s,1} = s$.  This value
is attained by the following polynomial in $\Pit_{s,1}$ 
$$\psi_{s,1} (z, \zt)= \left(1+\frac{z}{s}\right)^s,$$
which corresponds to performing $s$ iterated forward
Euler steps of size $h/s$.
\end{proposition}

The proposition can be proved by checking the
radius of absolute monotonicity and noticing that
it achieves the bound \eqref{incRtsp}.
Thus the optimal first-order perturbed methods for
linear problems are  the same as the optimal
unperturbed methods for linear problems.
\begin{proposition}\label{prop:rs2}
For $p=2$ we have  $\cRt_{s,2} = \sqrt{s(s-1)}$.  This value
is attained by the following polynomial in $\Pit_{s,2}$
\begin{align} \label{optpoly2}
\psi_{s,2} (z, \zt)& = \frac{2(s+r)-1}{2(s+r)}\left(1+\frac{z}{r}\right)^s + 
\frac{1}{2(s+r)}\left(1+\frac{\zt}{r}\right)^s,
\end{align}
 where $r=\cRt_{s,2}$.
\end{proposition}
Again, the proposition can be proved by checking the
radius of absolute monotonicity and noticing that
it achieves the bound \eqref{incRtsp}.

Some of the other optimal polynomials also have rational coefficients.
Two optimal degree-four fourth order polynomials we found are 
\begin{align*}
\psi_{4,4}^{1} (z, \zt)& = \frac{1}{3}\left(1+\frac{z}{r}\right)^2 + 
\frac{17}{48}\left(1+\frac{z}{r}\right)^4 +
\frac{14}{48}\left(1+\frac{z}{r}\right)^2\left(1+\frac{\zt}{r}\right)^2 +
\frac{1}{48}\left(1+\frac{\zt}{r}\right)^4\, , 
\end{align*}
and
$$ \psi_{4,4}^{2}(z, \zt)= \frac{7}{16} \left(1+\frac{z}{r}\right)^4+\frac{3}{8}  \left(1+\frac{z}{r}\right)^2 \left(1+\frac{\zt}{r}\right)^2+\frac{1}{6} 
   \left(1+\frac{z}{r}\right)^3 \left(1+\frac{\zt}{r}\right)+\frac{1}{48} \left(1+\frac{\zt}{r}\right)^4\, , $$
where $r=\cRt_{4,4}=2$.
Thus the optimal polynomial in $\Pit_{s,p}$ is in general not unique.

\begin{remark}
As noted already, not all polynomials of the form \eqref{dwtaylor1} can be realized as 
the stability function of a perturbed Runge--Kutta method \eqref{PRunge--Kutta} with $s$ stages.
For example, the polyomial \eqref{optpoly2} with $s=2$ is not the stability
function of any two-stage method (i.e., using only evaluations of 
$f(u_n), \ft(u_n), f(y_1), \ft(y_1)$).  
It can be realized as the stability function of a method that has three stages,
using evaluations of $f(u_n), \ft(u_n), f(y_1), \ft(y_2)$.  The difference 
in cost between such methods depends on the nature of $f, \ft$; see \cite{gottlieb2006}.
For this reason, we stress that the values
in Table \ref{dwtable} are only {\em upper bounds} on what can be achieved.
We do not pursue the topic further here.   \hfill $\square$
\end{remark}

\subsubsection{Optimal threshold factors for perturbations of specified Runge--Kutta methods}\label{sec:ex_optpert}
We have no general method for finding $\RLinopt(\bbA)$ nor a corresponding
method.  In this section we report results of some symbolic searches.
In the case of the second-order methods, due to the small number of free
parameters, it is not difficult to prove that the results below are truly
optimal.

\begin{example}\label{ex:2eo2}
We consider explicit perturbed second-order 2-stage Runge--Kutta methods
\begin{align}\begin{tabular}{c|cc}
$0$ & $0$ & $0$ \\
 $\alpha$ & $\alpha$ & $0$ \\
\hline 
$\bbA$ \phantom{\huge I}  & $1-\frac{1}{2 \alpha}$ & $\frac{1}{2 \alpha}$ 
\end{tabular}
\qquad \begin{tabular}{c|cc}
  & $0$ & $0$ \\
  & $\tilde a_{21}$ & $0$ \\
\hline
$\Apert$ \phantom{\huge I}  & $\bt_1$ & $\bt_2$ 
\end{tabular}\label{2eo2OptLin}.
\end{align}
For these methods, function \eqref{bistabfunc} can be expanded as
\begin{align}\phipert(z,\zt)= 1 + z + \frac{1}{2} z^2 + \beta_{11} z (z+\zt) + \beta_{1} (z+\zt)+ \beta_2 (z+\zt)^2\, , \label{2estord2}
\end{align}
where
\begin{align}\label{betas}
\beta_{11}= b^t \At e + \bt^t A e= \bt_2 a_{21} + b_2 \tilde a_{21} \, , \qquad  \beta_1= \bt^t e= \bt_1+ \bt_2 \, , \qquad \beta_2=\bt^t \At e=\bt_2 \tilde a_{21}\, . 
\end{align}
The polynomial \eqref{2estord2} is realizable (in the sense that it corresponds to a 2-stage Runge--Kutta method \eqref{2eo2OptLin}) if the first and last equations in \eqref{betas} can be solved for $\tilde a_{21}$ and $\bt_2$ in $\mathbb{R}$. A simple computation gives that a necessary condition is 
$\beta_{11}^2-2 \beta_2\geq 0$.  

With the help of   the symbolic computation program Mathematica, we have computed the largest $r$ such that \eqref{2estord2} is a.m. at $(-r, -r)$ and the polynomial is realizable (see  \cite{HigKetKocRepository}). We have obtained that the optimal perturbation, denoted by $\Apert_L$, satisfies $\bt_2=\tilde a_{21}=0$ and $\bt_1=\frac{1}{3} \left(\sqrt{7}-2\right)$. Note that for these values the stability function \eqref{2estord2} is independent of $\alpha$.  Furthermore, \begin{equation}\RLinopt(\bbA)=\frac{1}{3} \left(1+\sqrt{7}\right)\approx 1.21525\, . \label{RLinopt}
\end{equation}
Observe that $\RLinopt(\bbA)<  \cRt_{2,2}=\sqrt{2}$. 
The stability function \eqref{2estord2} for the  optimal perturbed method is 
$$\phit_{(\bbA, \Apert_L)}(z,\zt)= \frac{1}{9} \left(4+\sqrt{7}\right) \left(1+ \frac{z}{r}\right)^2+ \frac{1}{9} \left(5-\sqrt{7}\right) \left(1+ \frac{\zt}{r}\right) \, , \label{2estord2opt}
$$
where $r=\RLinopt(\bbA)$, the value given in  \eqref{RLinopt}. \hfill $\square$
\end{example}

\begin{example} \label{ex:4eo4} 
 We consider now perturbations of the classical fourth--order Runge-Kutta method, of the form
\begin{equation}
\begin{tabular}{c|cccc}
$0$ & $0$ & $0$ & $0$ & $0$ \\[1ex]
$\frac{1}{2}$ & $\frac{1}{2}$ & $0$ & $0$ & $0$ \\[1ex]
$\frac{1}{2}$ & $0$ &  $\frac{1}{2}$ &$0$ & $0$ \\[1ex]
$1$ & $0$ & $0$ & $1$ & $0$ \\[1ex]
\hline
$\bbA$ \phantom{\huge I} & $\frac{1}{6}$  & $\frac{1}{3}$& $\frac{1}{3}$ & $\frac{1}{6}$   
\end{tabular} \qquad \begin{tabular}{c|cccc}
 & $0$ & $0$ & $0$ & $0$ \\[1ex]
 & $0$ & $0$ & $0$ & $0$ \\[1ex]
 & $\tilde a_{31}$ & $0$ & $0$ & $0$ \\[1ex]
& $\tilde a_{41}$ & $\tilde a_{42}$ & $0$ & $0$ \\[1ex]
\hline
$\Apert$ \phantom{\huge I}  & $\bt_{1}$& $\bt_{2}$ & $0$ & $0$ 
\end{tabular}
\label{4eo4}
\end{equation}
We consider these perturbations because,
 in order to obtain a nonzero  SSP coefficient for nonlinear problems, the analysis done in \cite{higueras2005a} shows that only the entries $\tilde a_{31}$, $\tilde a_{41}$, $\tilde a_{42}$, $\bt_1$ and $\bt_2$ in $\Apert$ need be nonzero. To study SSP coefficients for the linear case, we have to analyze the perturbed stability function, that in this case  is of the form 
\begin{align}\phipert(z,\zt) = 1 + z + \frac{1}{2} z^2 +   \frac{1}{6} z^3+  \frac{1}{24} z^4+   \beta_{1} (z+\zt)   + \beta_{11} z (z+\zt)+ \beta_{21} z^2 (z+\zt)\,  \label{esfun:4eo4}
\end{align}
where
$$\beta_{1}= \bt_1+\bt_2\, , \qquad \beta_{11}= \frac{1}{6} \left(3 \,\bt_2+2 \,\tilde a_{31}+\tilde a_{41}+\tilde a_{42}\right)\, , \qquad \beta_{21}=  \frac{1}{12} \left(2 \,\tilde a_{31}+\tilde a_{42}\right)  \, . 
$$
Next, we construct the Taylor expansion of \eqref{esfun:4eo4} in terms of a general value $r$, and we compute the largest $r$ such that all the coefficients in the Taylor expansion are nonnegative; in this case, polynomial \eqref{esfun:4eo4} is always realizable (in the sense that it corresponds to a perturbation of the form  \eqref{4eo4}).
After some computations, we obtain a coefficient
$\RLin(\bbA,\Apert)\approx 1.66728$, that is the positive root of  the polynomial $15 \, x^4-4 \,x^3-12 \,x^2-24 \,x-24=0$, and the coefficients
$$\beta_{1}=\frac{ 7 \,r_0^3-2\, r_0^2-6 \,r_0-12 }{12} \, , \qquad  \beta_{11}=\frac{ 5\, r_0^2-2 \,r_0-6 }{12} \, , \qquad   \beta_{21}=\frac{r_0-1}{6}\, , 
$$
where $r_0=\RLin(\bbA,\Apert)$. With these values, the
perturbed stability function can be written as
$$\phipert(z,\zt) =
  \gamma_{01} \left(1+\frac{\zt}{r_0}\right) +\gamma_{11} 
   \left(1+\frac{z}{r_0}\right)\left(1+\frac{\zt}{r_0}\right)+
  \gamma_{21}
  \left(1+\frac{z}{r_0}\right)^2 \left(1+\frac{\zt}{r_0}\right) + \gamma_{40} \left(1+\frac{z}{r_0}\right)^4\, , 
$$
where  
$$\gamma_{01}=\frac{r_0 \left(2\, r_0^3-r_0^2-6\right)}{6} \, ,    \quad \gamma_{11}=\frac{r_0^2 \left(r_0^2+2 \,r_0-6\right)}{12} \, , \quad \gamma_{21}=\frac{r_0^3\, (r_0-1) }{6}  \, ,  \quad \gamma_{40}= \frac{r_0^4}{24} \, . $$ 
This perturbed stability function can be realized with the family of perturbations
\begin{subequations}\label{coef4eo4L}
\begin{align*}
\tilde a_{31} = \frac{1}{2} \left(2 \, r_0-2-\tilde a_{42}\right),\,   \tilde a_{41} =\frac{1}{2} \left(5  \,r_0^2-6 \,
   r_0-2-6  \,\bt_2\right),  \, \bt_1 = \frac{1}{12} \left(7  \,r_0^3-2  \,r_0^2-6  \,r_0-12-12  \,\bt_2\right). 
\end{align*}
\end{subequations}
Observe that $\RLin(\bbA,\Apert)$ is independent of the choice of $\tilde a_{42}$ and $\bt_2$.  Thus, for $\tilde a_{42}=\bt_2=0$, we obtain the same value of $\RLin(\bbA,\Apert)$ with a perturbation \eqref{4eo4} whose nontrivial elements are only in the first column of $\Apert$.

Observe too that the perturbation in \eqref{4eo4} does not contain all the possible nonnegative elements in a strictly lower triangular matrix (see Definition \ref{def:pertRK}), and therefore we cannot claim that the value $\RLin(\bbA,\Apert)\approx 1.66728$ is  the threshold factor of the optimal  perturbation $\RLinopt(\bbA)$. With the study done, we have that $1.66728 \leq \RLinopt(\bbA)\leq  \cRt_{4,4}=2$. \hfill $\square$

\end{example}

\section{Perturbed Runge--Kutta methods for nonlinear problems}\label{sec:nonlinear}

In this section we seek to answer the questions posed in 
Section \ref{questions}  for nonlinear problems. We begin with an introduction section where we collect some known results from the literature.   
 
\subsection{Introduction} 
In this section the introduce some notation used in the rest of the paper and we collect some known results from the literature. 

First, it is convenient to write  
scheme \eqref{Runge--Kutta} in canonical Shu-Osher form \cite{SSPbook}
\begin{align} \label{cso}
    Y & = v_r u_n + \alpha_r \left( Y + \frac{h}{r} F\right)
\end{align}
where
\begin{align} \label{canonical-coefficients}
   v_r       = (I + r\bbA)^{-1} e\, , \qquad 
     \alpha_r   = r(I + r\bbA)^{-1} \bbA\, .
\end{align}
Observe that matrices $\bbA$ and $\alpha_r$ have the same structure (strictly lower triangular, lower triangular or full).

\begin{definition} \label{def:absmon}
The {\it radius of absolute monotonicity} of 
a Runge--Kutta method \eqref{Runge--Kutta} is the largest $r$ such that $v_r$ and $\alpha_r$  in \eqref{canonical-coefficients} exist and are non-negative:
\begin{align}R(\bbA)=\sup\left\{r\, | \, r=0 \text{ or } r>0, (I + r\bbA )^{-1} \hbox{ exists, and }   \alpha_r,v_r\ge 0 \right\}\, . \label{ramdef0}
\end{align}
\end{definition}
\noindent Recall that Definition \ref{def:absmon} is the one in \cite[Def. 2.4]{Kraaijevanger1991} using the notation given in \cite[Eq. (1.21)]{higueras2005a}. The quantity $R(\bbA)$ is also known as the SSP coefficient or
Kraaijevanger coefficient.
As usual, the inequalities above should be understood component--wise.

In \cite[Thm. 2.5]{ferracina2004} step size restrictions to obtain monotonicity are given in terms of the radius of absolute monotonicity of the method. Thus, the larger $R(\bbA)$ is, the larger is the step size restriction for monotonicity; in particular, if $R(\bbA)=0$, numerical monotonicity cannot be ensured.

Next, we consider perturbed Runge--Kutta methods \eqref{PRunge--Kutta}. 
To study absolute monotonicity of perturbed Runge--Kutta
methods, we write method \eqref{PRunge--Kutta} also in a canonical
Shu-Osher-like form
\begin{align} \label{dwRunge--Kutta_cso}
Y & = \gamma_r u_n + \alphaup_r  \left(Y + \frac{\Dt}{r}F\right) + 
                    \alphadown_r \left(Y - \frac{\Dt}{r}\Ft\right)\, , 
\end{align}
where
\begin{subequations} \label{pert-cso-coeff}
\begin{align}
\gamma_r & = (I + r\bbA + 2r\Apert)^{-1} e\,  ,\label{pert-cso-coeffa}\\[0.5ex]
\alphaup_r & = r(I + r\bbA + 2r\Apert)^{-1} (\bbA + \Apert)\, , \label{pert-cso-coeffb} \\[0.5ex] 
\alphadown_r & = r(I + r\bbA + 2r\Apert)^{-1} \Apert\, . \label{pert-cso-coeffc} 
\end{align}
\end{subequations}
Observe that method  \eqref{dwRunge--Kutta_cso}, with $ \gamma_r=(I-\alphaup_r-\alphadown_r)e$, is a perturbed Runge--Kutta scheme with Butcher coefficients 
\begin{align}\bbA=\frac{1}{r}(I-\alphaup_r-\alphadown_r)^{-1} (\alphaup_r-\alphadown_r)\, , \qquad \Apert=\frac{1}{r}(I-\alphaup_r-\alphadown_r)^{-1} \alphadown_r\, , \label{AApert} \end{align}
provided that $(I-\alphaup_r-\alphadown_r)^{-1}$ exists. 

\begin{definition}\label{def_radius_pert} \cite[Def. 3.1]{higueras2005a}
The {\it radius of absolute monotonicity} of a perturbed Runge--Kutta method  $(\bbA,\Apert)$ is
the largest $r$ such that $\gamma_r$, $\alphaup_r$ and $\alphadown_r$ in \eqref{pert-cso-coeff} exist and are non-negative:
\begin{align}R(\bbA,\Apert)=\sup\left\{r\, | \, r=0 \text{ or } \, r>0, \,(I + r\bbA + 2r\Apert)^{-1} \hbox{ exists, and }   \gamma_r,\, \alphaup_r,\, \alphadown_r\ge 0 \right\}\, . \label{ramdef}
\end{align}
\end{definition}

 For perturbation $(\bbA, \Apert)$, step size restrictions to obtain monotonicity are given in terms of $R(\bbA,\Apert)$ \cite[Thm. 3.5]{higueras2005a},  
 \begin{align}\label{eq:stepsize_r}
h\leq  R(\bbA,\Apert) \, h_0\,. 
 \end{align} 
Thus, perturbations with large values of $R(\bbA, \Apert)$ ensure larger step size restrictions for monotonicity.

 \begin{remark} {\it (Fictitious perturbations)}  \label{remark:ficticius}
As it has been pointed out in Section \ref{sec:example}, if  function $f$ in  \eqref{ode} satisfies both  \eqref{FEcond1} and \eqref{FEcond2}, we can
formally introduce a function $\ft=f$  to perturb fictitiously the Runge--Kutta method (see \eqref{PRunge--Kutta}). In this way, the standard (unperturbed) Runge--Kutta method can be written as \eqref{dwRunge--Kutta_cso}.  Thus, the results in this paper can also be used to ensure monotonicity for step size restrictions larger than the ones given in terms of the (unperturbed) SSP coefficient. \hfill $\square$
\end{remark} 

\begin{remark} \label{propertyC}  {\it (Property C)}   Most previous works, including \cite{ruuth2004,ruuth2006}, have focused
on methods with the following property:	
for each value of $j$
    \begin{align}\label{eqpropC}
    	\Kt_{ij}   \ne 0 \mbox{ (for some $i$)} \quad \implies  \quad
                  K_{ij}   = 0 \mbox{ (for all $i$).}
    \end{align}
In this case, we will say   that a perturbation $\Kt$ to a Runge--Kutta method $K$ possesses
    {\em property C}.  In words, property C means that in the $j$th column, only one of $K, \Kt$
has any nonzero entries.  Thus, only one of $f(y_j), \ft(y_j)$
need ever be evaluated, so only $s$ total function evaluations are required per step.
In \cite{gottlieb2006} it was shown that
for WENO discretizations, the cost of computing both
$f(y_j)$ and $\ft(y_j)$ is much less than twice the cost of computing $f(y_j)$
alone.  
Therefore methods without property C may also be of practical
interest.  In the present work, we do not assume property C. \hfill $\square$
\end{remark}

 \subsubsection{Zero-well-defined perturbations}
Regularity of $(I-\alphaup_r-\alphadown_r)$ is evidently important
in our study. Observe that from \eqref{pert-cso-coeff} we have
\begin{equation}\label{zerowell}
(I-\alphaup_r-\alphadown_r)(I+r \bbA)=(I-2 \alphadown_r)\, . 
\end{equation}
Consequently, if $I+r \bbA$ is regular for some $r$, then $(I-\alphaup_r-\alphadown_r)$ 
is regular if and only if $(I-2 \alphadown_r)$ is regular.

If $I-2\alphadown_r$ is singular, then the stage equations
do not have a unique solution even for the trivial ODE given
by $f=0$.  This motivates the following deinition.
\begin{definition}
Let a perturbed Runge--Kutta method \eqref{dwRunge--Kutta_cso} be given.
If $I-2\alphadown_r$ 
(defined by \eqref{pert-cso-coeffc}) is non-singular,
we say that the perturbation is {\it zero-well-defined}.  
\end{definition}

See \cite[Chap. 3]{SSPbook}
for the analogous definition in the context of traditional Runge--Kutta methods.

\subsection{Optimal perturbations}
In this section we answer question \ref{Q1} of Section
\ref{questions} by showing that every method can be perturbed so as to 
give a method with strictly positive SSP coefficient.
\begin{theorem}\label{thm:feasible}
Let $\bbA$ be a Runge-Kutta method that belongs to a specified class of methods 
(explicit, diagonally implicit, or fully implicit).  Then it is always possible to find a perturbation $\Apert$ within the same class such that $R(\bbA,\Apert)>0$.
\end{theorem}     

Thus it makes sense to deal with perturbations that give the largest SSP coefficient. We formalize this idea in the following definition.

\begin{definition}\label{def_opt_radius_pert}
The {\it optimal perturbed SSP coefficient} of a Runge--Kutta method $\bbA$ is
denoted by
$$\Rt(\bbA)=\sup_{\Apert} R(\bbA,\Apert)\, . 
$$
For a given method $\bbA$ that is (explicit/diagonally implicit/fully implicit), 
we consider the supremum  over perturbations $\Apert$ that are zero-well-defined and
correspond to the same class of methods.
A matrix $\Apert$ such that $R(\bbA,\Apert)=\Rt(\bbA)$ is called an
{\it optimal perturbation}.  
\end{definition}

Observe that a perturbed Runge--Kutta method $(\bbA,\Apert)$ can be interpreted as an
additive Runge--Kutta method $(\bbA+\Apert,\Apert)$ for functions $(f, \ft)$ (see Remark \ref{rem:additive}),
and conditions \eqref{pert-cso-coeff} are the ones required for the absolute
monotonicity of this additive scheme at $(z, \zt)=(-r, -r)$ (see
\cite{higueras2006}). From Lemma 2.8 in \cite{higueras2006}, we obtain that
the stability function $\phipert$ defined by \eqref{bistabfunc},
is  absolutely monotonic at $(\xi, \tilde \xi)=(-r, -r)$. Consequently,
\begin{equation} R(\bbA,\Apert)\leq \RLin(\bbA,\Apert)\leq \RLinopt(\bbA) \, . \label{bound1}
\end{equation}
Furthermore, from  SSP theory and inequality \eqref{bound:RRL},  we have
\begin{equation}R(\bbA)\leq R(\phi_K)\leq \RLinopt(\bbA)\, , \qquad  R(\bbA)\leq  \Rt(\bbA)\leq \RLinopt(\bbA)\, . \label{bound2}
\end{equation}
The following example illustrates that $R(\phi_K)$ can be either larger
or smaller than $\Rt(\bbA)$.  

\begin{example}\label{ex_2stage2order}
We consider the family of second order 2-stage Runge-Kutta methods \eqref{2eo2OptLin} for $\alpha\in\mathbb{R}$. For this family we have
\begin{align*}
v_r\geq 0 \quad & \Longleftrightarrow \qquad \alpha>0 \qquad \hbox{and}\quad 0\leq r\leq\frac{1}{\alpha}\, , \\[0.5ex]
\alpha_r\geq 0 \qquad & \Longleftrightarrow \qquad \alpha\geq\frac{1}{2} \quad \hbox{and}\quad 0\leq r\leq\frac{2\alpha-1}{\alpha}\, . 
\end{align*}
Thus
\begin{equation}\label{R2stageorder2}
R(\bbA)=\begin{cases} 0, & \hbox{if} \quad \displaystyle  \alpha\leq \frac{1}{2}\, , \\[1ex]
\displaystyle  \frac{2 \, \alpha -1}{\alpha }\, ,  & \displaystyle \hbox{if}   \quad\frac{1}{2}<\alpha \leq 1\, ,\\[1ex]
\displaystyle   \frac{1}{\alpha }\, ,  & \hbox{if}   \quad 1<\alpha\, . 
\end{cases}
\end{equation}
In Figure \ref{f22} we show the threshold factor $R(\phi_\bbA)$ (thin solid blue line) and the SSP coefficient  $R(\bbA)$ (thick solid black line). We also show the corresponding optimal coefficients for perturbed methods, namely,  the optimal threshold factor $\RLinopt(\bbA)$  (thin dashed blue line) given by  \eqref{RLinopt} in Example \ref{ex:2eo2}, and the optimal SSP coefficient $\Rt(\bbA)$ for the perturbed method (thick dashed black line) given by \eqref{radiopert22}.

We see that for optimal SSP method ($\alpha=1$) it is not possible to increase the SSP coefficient by means of perturbations. However, for $\alpha=(\sqrt{7}-1)/2$ it is possible to obtain a perturbation that raises the SSP coefficient to $\Rt(\bbA)=\RLinopt(\bbA)=\left(1+\sqrt{7}\right)\approx 1.21525$ (see \eqref{RLinopt}). 

We have that $R(\phi_\bbA)=\Rt(\bbA)=1$ for $\alpha=2/3, 1$.  For  $2/3< \alpha<1$ we obtain 
$R(\phi_\bbA) <\Rt(\bbA) 
$, whereas for $0<\alpha < 2/3$  and for $1< \alpha$  
we have $\Rt(\bbA)  < R(\phi_\bbA) $.

Coefficients of the perturbations that give rise to these values are given in
Appendix \ref{sec:2stage-details}.  \hfill $\square$
\end{example}

\begin{figure}\begin{center}\includegraphics[scale=0.45]{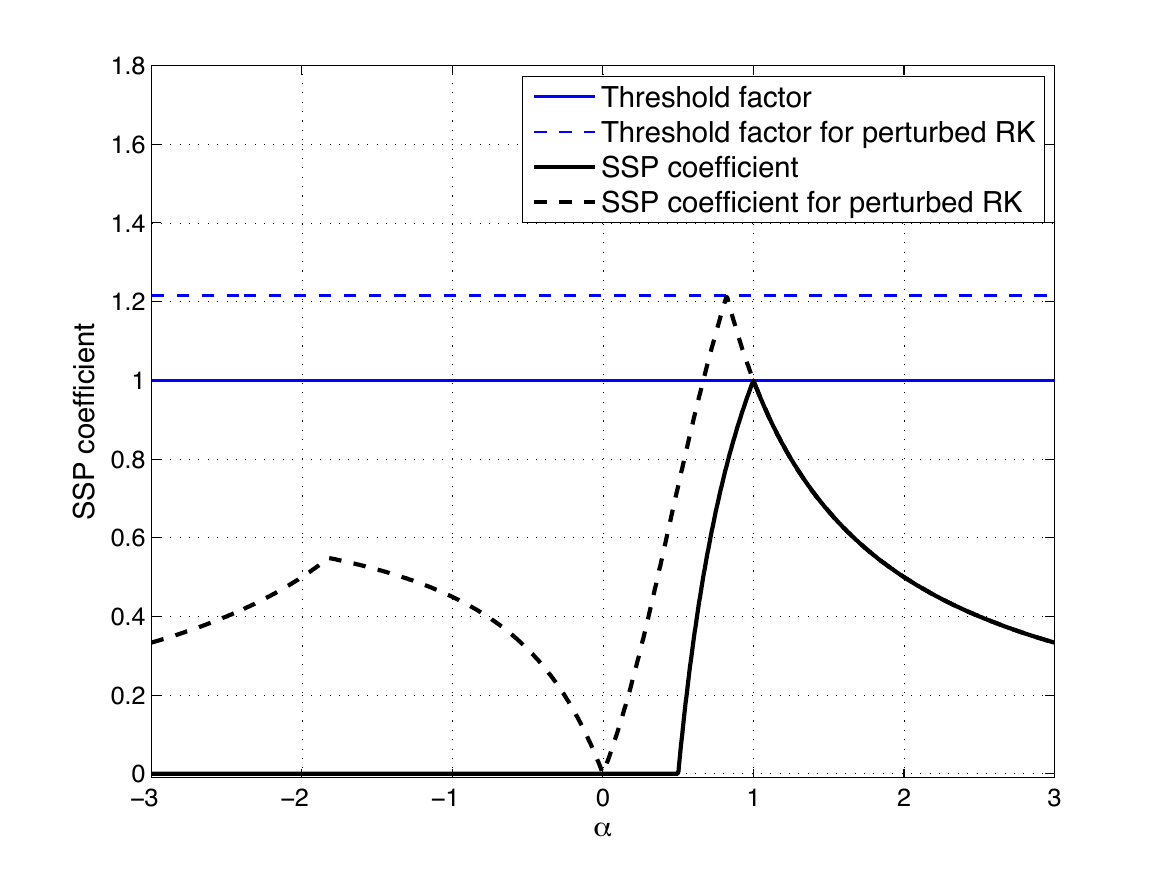}
\caption{Family of second order 2-stage methods: SSP coefficients for unperturbed methods and optimal SSP coefficients for perturbed methods.}
\label{f22}
\end{center}
\end{figure}

\subsection{Upper bounds on the SSP coefficient for perturbed Runge--Kutta methods\label{sec:nonlinear-bounds}}
In this section we answer questions \ref{Q2}  and \ref{Q3} of Section \ref{questions}
We begin by exploring some upper bounds on the SSP coefficient $\Rt(\bbA)$ where $\bbA$ is an $s$-stage order $p$ Runge--Kutta method. 
A straightforward upper bound is obtained from inequality \eqref{bound2} and Theorem \ref{incRtsp}:
\begin{align} \Rt(\bbA)\leq \sqrt[p]{s(s-1)\ldots (s-p+1)}\,  .\label{bound4}
\end{align}
Another bound is given by the next Theorem. 
\begin{theorem}\label{prop_vpos}
Consider an explicit Runge-Kutta method $K$  and let $r_e$ be the largest positive value such that vector $v_r$ in \eqref{canonical-coefficients} is non-negative. Then
\begin{equation}
\Rt(\bbA)\leq r_e\, . 
\end{equation}
\end{theorem}

From Theorem \ref{prop_vpos} we obtain that
\begin{equation}
{R}(\bbA)\leq \Rt(\bbA)\leq r_e\, . \label{in_re}
\end{equation}
Consequently, for those methods such that ${R}(\bbA)=r_e$, the SSP coefficient cannot be increased by perturbation. This is the case for the family of second-order two-stage methods. 
For $\alpha\geq 1$, ${R}(\bbA)= r_e=1/\alpha$ (see Example \ref{ex_2stage2order}). 

On the other hand, if ${R}(\bbA)<r_e$ one can try to find a perturbation to increase the SSP coefficient. This is the case for the classical 4-stage order 4 method for which ${R}(\bbA)=0$ and  $r_e\approx 1.2956$, the real root of $x^3 - 2 x^2 + 4 x - 4=0$.

Finally, another interesting bound, for explicit methods only, can be obtained in terms of the 
Butcher coefficients of the Runge-Kutta method $\bbA$. 
\begin{theorem} \label{Pbound} Consider an explicit Runge-Kutta method $\bbA$ with perturbed SSP coefficient $\Rt(\bbA)>0$. Let $\bbA=(a_{ij})$.  Then 
\begin{align}\Rt(\bbA)\leq  \frac{1}{\max_{ij} |a_{ij}|}\, . \label{Rpert}
\end{align}
\end{theorem}

Consequently, $$R(\bbA)\leq  \Rt (\bbA)\leq \frac{1}{\max_{ij} |a_{ij}|}\,  .$$ 
For those methods such that
$R (\bbA)= 1/{\max_{ij} |a_{ij}|}
$
it is not possible to increase the SSP coefficient by perturbing the method.
This is the case for all known optimal explicit SSP Runge--Kutta methods of orders one through
three, with any number of stages \cite{ketcheson2008}.

For the restricted class of perturbations considered in \cite{ruuth2004},
similar results were obtained in  \cite[Thms. 3.1, 3.4, 3.5 and
3.6]{ruuth2004}. Theorem \ref{Pbound} extends those results, showing that
no improvement in the radius of absolute monotonicity is possible for many 
optimal SSP methods, even when more general perturbations are considered.
Those methods include all optimal methods of order one or two, the optimal methods
of order three with $n^2$ stages (for any integer $n$), and the optimal methods
with $(s,p) \in \{(5,3), (6,3), (10,4) \}$.  Interestingly, the widely-used
optimal five-stage fourth-order method is an exception; it has 
$R(K) \approx 1.508$ while the Theorem \ref{Pbound} gives an upper bound of
$1.8349...$.  Numerical computations suggest that it can be perturbed to
achieve $\Rt(K) \approx 1.63979$.

\subsection{Relations among the Butcher and canonical Shu-Osher representations}\label{sec:ButcherShu}
To answer questions \ref{Q3} and \ref{Q4} in Section
\ref{questions}, two algorithms are proposed. In this section we prove some
technical results that justify some steps in Algorithms \ref{alg:explicit-LP}
and \ref{alg:2}  below.

For explicit Runge--Kutta methods (or other methods with one stage equal to $u_n$), some components of vector $v_r$  in \eqref{canonical-coefficients} 
can be moved to the first column of matrix $\alpha_r$ \cite[Remark 1]{higueras2005a}. This simple transformation may yield a larger value of $R(\bbA)$ -- and never yields a smaller value--. A similar transformation, that may give a larger
value of $R(\bbA,\Apert)$ but never smaller values,  exists for perturbations of this class of schemes.
\begin{proposition} \label{propfirstcol}
	Let an $s$-stage explicit perturbed Runge--Kutta method with coefficients 
    $\gamma_r, \alphaup_r, \alphadown_r \ge 0,$ be given where $r$ is the radius of absolute
    monotonicity of the method.  Consider the perturbed method with coefficients
    \begin{subequations} \label{transform}
    \begin{align}
    	\widehat{\gamma} & = (1, 0 , \ldots, 0)^t\, ,  \\
        \widehat{\alpha}^\textup{up}_{i,1} & = \alphaup_{i,1} + (\gamma_r)_{i}/2 & 2 \le i \le s  \\
        \widehat{\alpha}^\textup{down}_{i,1} & = \alphadown_{i,1} + (\gamma_r)_{i}/2 & 2 \le i \le s .
    \end{align}
    \end{subequations}
    Then the perturbed method with coefficients $(\gamma_r, \alphaup_r, \alphadown_r)$
    and the modified perturbed method with coefficients $(\widehat{\gamma}, 
    \widehat{\alpha}^\textup{up}, \widehat{\alpha}^\textup{down})$
    correspond to the same Runge--Kutta method $\bbA$.  The modified perturbed method has radius of absolute
    monotonicity at least equal to $r$.
\end{proposition}

\begin{remark}Proposition \ref{propfirstcol} is also valid for Runge-Kutta methods whose first row is equal to zero.\hfill $\square$
\end{remark}

In the Butcher form \eqref{PRunge--Kutta} it is obvious which
perturbed methods $(K,\Kt)$ correspond to a given method $K$.
In the canonical Shu-Osher form it is less obvious.
The following lemma characterizes which methods of the form \eqref{dwRunge--Kutta_cso}
are perturbations of a given method \eqref{cso}.
\begin{lemma}\label{equivprop}
If method \eqref{dwRunge--Kutta_cso} is a perturbation of method \eqref{cso}, then 
their coefficients are related as follows:
\begin{subequations} \label{alpharel1}
\begin{align} 
 (I-2\alphadown_r)\alpha_r & = (\alphaup_r - \alphadown_r) \label{alpharel1a} \\
 (I-2\alphadown_r)v_r & = \gamma_r.
\end{align}
\end{subequations}
Furthermore, if \eqref{alpharel1} holds and 
the perturbation is zero-well-defined,
then \eqref{dwRunge--Kutta_cso} is a perturbation of \eqref{cso}.
\end{lemma}
 
Lemma \ref{equivprop} does {\em not} imply that the perturbation \eqref{dwRunge--Kutta_cso}
is unique for a given $r$; see Proposition \ref{propfirstcol}.

\begin{remark} \label{invertibility-remark}
The necessity of the zero-well-defined condition in the second part of
Lemma \ref{equivprop} can be seen from the following example.
We take the implicit trapezoidal Runge-Kutta method
\begin{align*}
 \begin{array}{c|cc}
0 & 0 & 0 \\
1/2 & 1/2 & 1/2 \\ \hline
    & 1/2 & 1/2
\end{array}.
\end{align*}
The canonical form \eqref{cso} is then
\begin{align*}
\alpha_r & = \begin{pmatrix}
0 & 0 & 0 \\[1ex]
\frac{r}{r+2} & \frac{r}{r+2} & 0 \\[1ex]
\frac{r}{r+2} & \frac{r}{r+2} & 0
\end{pmatrix}\, ,
& v_r & = \begin{pmatrix} 1 \\[1ex] \frac{2-r}{r+2} \\[1ex] \frac{2-r}{r+2} \end{pmatrix}.
\end{align*}
Then \eqref{alpharel1} is satisfied -- for any $r$ -- by
\begin{align*}
\alphaup = \alphadown & = \begin{pmatrix} 
1/3 & 0 & 0 \\
0 & 1/2 & 0 \\
0 & 1/2 & 0
\end{pmatrix}\, , &
\gamma & = \begin{pmatrix} 1/3 \\ 0 \\ 0 \end{pmatrix}.
\end{align*}
However, this method -- which involves a perturbation that is not zero-well-defined 
-- is not a perturbation of the original method. \hfill $\square$
\end{remark}

\subsection{Computing optimal perturbations\label{sec:algorithms}}
In this section we present two algorithms to symbolically
or numerically find the optimal perturbed SSP coefficient and a corresponding perturbation of a given Runge--Kutta method.  The first algorithm is proven to approximate the
optimal value to any accuracy, contingent on the computational solution of
linear program subproblems.  It is only valid for explicit perturbations.   The second
algorithm uses no floating-point approximations, and can
be applied to both explicit and implicit
methods, but it is not proven to give the optimal value. 
  The results of the
two algorithms coincide (to high precision) for all explicit methods on which 
we have tested them.

\subsubsection{Provably correct algorithm for finding optimal explicit perturbations}
In the foregoing, we have shown that finding an optimal perturbation
consists of determining the largest $r$ such that there exists a splitting
satisfying \eqref{alpharel1} with positive coefficients.
Note that the range of
values for which a method $(K,\Kt)$ is absolutely monotonic is
always the interval $[0, R(K,\Kt)]$.  Therefore, one way to
find the largest $r$ is to devise a method for testing for
a given $r$ whether there exists a perturbation $\Kt$ such that $R(K,\Kt)\ge r$.
For given method \eqref{Runge--Kutta} and value of $r$,
the system of equations \eqref{alpharel1} together with the inequalities
$\alphaup_r, \alphadown_r \ge 0$
constitutes a linear programming (LP) feasibility problem.
The following theorem is an immediate consequence of Lemma \ref{equivprop}.

\begin{theorem}\label{Theorem:5}
Let an $s$-stage Runge--Kutta method $\bbA$ and a positive number $r$ be given.  There exists a 
perturbation $\Apert$ with $R(\bbA,\Apert)\ge r$ if and only if there exists
an $(s+1)\times(s+1)$ matrix $\alphadown_r$ such that $(I-2\alphadown_r)$ is regular and
the following componentwise inequalities hold:
\begin{subequations} \label{LP}
\begin{align} 
 (I-2\alpha^{down}_r)\alpha_r + \alpha^{down}_r & \ge 0 \\
 (I-2\alpha^{down}_r)v_r & \ge 0 \\
 \alphadown_r \ge 0.
\end{align}
\end{subequations}
\end{theorem}
The linear program \eqref{LP} can be solved by standard LP solvers.
By embedding this solution in a one-dimensional root-finding algorithm,
optimal perturbations can be found.  An algorithm based on bisection 
is given as Algorithm \ref{alg:explicit-LP}.   For a prescribed
tolerance $\epsilon$, it returns a value that is less than
or equal to optimal perturbed radius of absolute monotonicity, 
and is within $\epsilon$ of that value.
\begin{algorithm}\caption{Optimal explicit perturbation}
\label{alg:explicit-LP}
\begin{algorithmic}
\INPUT $K, \epsilon$
\State $\rmax := 1/\max|a_{ij}|, \rmin := 0$.
\While {$\rmax - \rmin > \epsilon$}
	\State $r=(\rmax + \rmin)/2$.
	\State Compute the coefficient matrices $\alpha_r, v_r$ using \eqref{canonical-coefficients}.
	\State Solve the LP given by \eqref{LP}.
    \If {it is feasible} 
    	\State $\rmin:=r$  
    \Else
    	\State $\rmax:=r$
    \EndIf
\EndWhile \\
\Return $\rmin$
\end{algorithmic}
\end{algorithm}

Assuming the solution of the LP is correct, the algorithm provably
finds an optimal explicit perturbation.  However,
for implicit perturbations the LP solver may converge to a solution
(like the method in Remark \ref{invertibility-remark} above)
for which $I-2\alphadown_r$ is singular.

\subsubsection{Iterated splitting algorithm}
We next investigate how to choose $\alphaup_r, \alphadown_r$ directly so as
to find a perturbation with radius of a.m. at least $r$.
The following result suggests an approach.
\begin{lemma} \label{signlemma}
Given an explicit Runge--Kutta method \eqref{cso},
let $\alphaup_r\ge0, \alphadown_r\ge0$ denote coefficients of a 
zero-well-defined perturbation of \eqref{cso}.
Then there exist matrices $\alpha^+\ge 0$, $\alpha^-\ge 0$,  
 such that
\begin{align}  \label{alphar}
    \alphaup_r       = (I+2\alpha^-)^{-1} \alpha^+\, , \qquad  \qquad 
    \alphadown_r     = (I+2\alpha^-)^{-1} \alpha^-,
\end{align}
and $\alpha_r=\alpha^+ - \alpha^-$.
\end{lemma}

Thus, zero-well defined perturbations with $\alphaup_r\ge0, \alphadown_r\ge0$ come from splittings of the matrix $\alpha_r$ and expressions \eqref{alphar}.

In the next algorithm we use the following notation:
\begin{align*}
((x)^+)_{ij} & = \begin{cases}
x_{ij} & \mbox{ if } x_{ij}\ge0 \\
0 & \mbox{ if } x_{ij}<0.  \end{cases} &
((x)^-)_{ij} & = \begin{cases}
0 & \mbox{ if } x_{ij}\ge0 \\
-x_{ij} & \mbox{ if } x_{ij}<0\, ,   \end{cases}
\end{align*}
and thus $x=(x)^+-(x)^-$ is a sign splitting of matrix $x$, with $(x)^+\geq 0$, $(x)^-\geq 0$. 

Given a perturbed Runge-Kutta method \eqref{dwRunge--Kutta_cso} with $\gamma_r=e_1$, where $e_1=(1, 0, \ldots, 0)^t$, and $\alphaup$ or $\alphadown$ containing negative values,  
 we construct 
\begin{subequations}\label{newpert}
\begin{align}
\tilde \gamma_r&=\left(I+2 \, (\alphaup_r)^-+ 2 \,(\alphadown_r)^-\right)^{-1}  e_1\, ,\\[0.5ex]
\tilde\alpha_r^\textup{up}&=\left(I+2 \,(\alphaup_r)^-+ 2 \,(\alphadown_r)^-\right)^{-1} \left((\alphaup_r)^++ (\alphadown_r)^-\right)\, ,   \\[0.5ex]
\tilde\alpha_r^\textup{down}&=\left(I+2 \,(\alphaup_r)^-+ 2 \,(\alphadown_r)^-\right)^{-1}  \left((\alphaup_r)^-+(\alphadown_r)^+\right)\, ,  
\end{align}
\end{subequations}
where 
$\alphaup=(\alphaup_r)^+-(\alphaup_r)^-$, $\alphadown_r=(\alphadown_r)^+-(\alphadown_r)^-$, provided that $I+2 (\alphaup_r)^-+ 2 (\alphadown_r)^-$ is invertible. Using Lemma \ref{equivprop}, it is straightforward to prove that, if method $\alphaup_r$, $\alphadown_r$ is a perturbation of   \eqref{cso}, then \eqref{newpert} is also perturbation of \eqref{cso}. 

Next, for explicit methods, we perform transformation \eqref{transform}. 
In this way, \eqref{newpert} followed by transformation \eqref{transform} gives a perturbation of the form \eqref{dwRunge--Kutta_cso} with $\gamma_r=e_1$,  that we denote by  $\hat \alpha_r^\textup{up}$, $\hat \alpha_r^\textup{down}$. If $\hat \alpha_r^\textup{up}\geq 0$,  $\hat \alpha_r^\textup{down}\geq 0$, then $r$ is an SSP coefficient; otherwise, we can repeat the above process.  

The following lemma studies the sign of $(\hat \alpha_r^\textup{up})_{ij}$, $(\hat \alpha_r^\textup{down})_{ij}$ when   $(\alphaup_r)_{ij}<0$ or $(\alphadown_r)_{ij}<0$. For the sake of clarity, we drop the index $r$. 

\begin{lemma}\label{lemma_row} We consider a perturbed explicit Runge-Kutta method  with coefficients $\gamma=e_1$, $\alphaup$, $\alphadown$, and 
the perturbation   $\hat \alpha^\textup{up}$, $\hat \alpha^\textup{down}$ obtained by computing \eqref{newpert} followed by transformation \eqref{transform}.  Assume that  $j_0\geq 2$ is the first row with negative terms in $\alphaup$ or $\alphadown$. Let $m_0$ be the largest index $m_0\geq 1$ such that  $\alphaup_{j_0,m_0}<0$ or $\alphadown_{j_0,m_0} <0$. Then
\begin{enumerate}
\item For first to $(j_0-1)$-th row, we have: $\hat \alpha^\textup{up}_{i,j}=\alphaup_{i,j}$ and $\hat \alpha^\textup{down}_{i,j}=0$ for $1\leq i\leq j_0-1$, $1\leq j\leq j_0-2$. 
\item For the $j_0$-th row, we have:
\begin{enumerate}
\item If  $m_0=1$, then $\hat \alpha^\textup{up}_{j_0,1}<0$ or $\hat \alpha^\textup{down}_{j_0,1}<0$. 
\item If $m_0\geq 2$, then,  $\hat \alpha^\textup{up}_{j_0,m_0}\geq 0$  and $\hat \alpha^\textup{down}_{j_0,m_0}\geq 0$. 
\item For $1\leq m_0\leq j_0-2$, we have $\hat \alpha^\textup{up}_{j_0,\ell}\geq 0$  and $\hat \alpha^\textup{down}_{j_0,\ell}\geq 0$ for $\ell=m_0+1, \ldots, j_0-1$. 
\end{enumerate}
\end{enumerate}
 
 \end{lemma}
 
Consequently, if matrices $\alphaup_r$ and $\alphadown_r$  contain negative elements in the second or later columns, an iterated construction of perturbations $\hat \alpha_r^\textup{up}$, $\hat \alpha_r^\textup{down}$ removes these negative values obtaining a perturbation with non-negative elements from second column on. However, if in a row $j_0$ we have:
\begin{equation}\label{stop}
  \alphaup_{j_0,1}<0 \, \qquad \hbox{and} \qquad \alphaup_{j_0,\ell}\geq 0 \qquad \ell=2, \ldots, j_0-1\, ,
\end{equation}
or
\begin{equation}\label{stop2}
  \alphadown_{j_0,1}<0 \, \qquad \hbox{and} \qquad \alphadown_{j_0,\ell}\geq 0 \qquad \ell=2, \ldots, j_0-1\, ,
\end{equation}
 the new perturbation $\hat \alpha_r^\textup{up}$, $\hat \alpha_r^\textup{down}$ will also contain negative elements in the first column.

We now give Algorithm \ref{alg:2} to determine whether there exists a perturbation 
with a.m. radius $r$ for a given method.
\begin{algorithm}\caption{Existence of a perturbation with radius $r$}
\label{alg:2}
\begin{algorithmic}
\INPUT $r, K$
    \State Compute the coefficient matrices $\alpha_r, v_r$ using \eqref{canonical-coefficients}.

    \State Set $\alphaup = \alpha_r$ and $\alphadown = 0$.
	\While {$\alphaup$ or $\alphadown$ has any negative entries}
	    \State If $K$ has a zero row, perform the transformation \eqref{transform}.
	    \State If $\alphaup, \alphadown \ge 0$, stop.  This is a feasible perturbation.
        \State If condition \eqref{stop} or \eqref{stop2} hold, stop. A feasible perturbation cannot be found.   
        \State Set $\alpha^- = (\alphaup)^- + (\alphadown)^+$ and 
    		  $\alpha^+ = (\alphaup)^+ + (\alphadown)^-$
        \State Compute a new splitting:
    		\begin{align*}
            	\alphaup & = \left(I + 2 ((\alphaup)^- + (\alphadown)^-)\right)^{-1}\alpha^+ \\
                \alphadown & =\left(I + 2((\alphaup)^- + (\alphadown)^-)\right)^{-1}\alpha^-
            \end{align*}
\EndWhile \\
\end{algorithmic}
\end{algorithm}

\begin{remark}
The difficulty in proving the correctness of Algorithm \ref{alg:2}
for explicit
methods is that one could use
$(\alpha_r)^+ + \delta$, $(\alpha_r)^- + \delta$,
in place of $(\alpha_r)^+, (\alpha_r)^-$, where $\delta$ is any non-negative matrix. \hfill $\square$
\end{remark}

\subsection{Examples and numerical tests}\label{sec:examples}

\subsubsection{Examples}
In this section we compute optimal perturbations of some existing methods,
using the algorithms described in the last section.  

\begin{table}
\begin{tabular}{ll|l|l|lll|lll}
\toprule
Order & Stages & Method & $R(K)$ &    $\Rt(K)$   & Bound  & Bound  & Property C \\
& & &   &   &  \eqref{Rpert} &\eqref{bound4}  &  \eqref{eqpropC}  \\
\midrule
1 & 1 & Forward Euler & 1 &            1 &        1 &       1 &                  True \\
\midrule
2 & 2 & Midpoint         & 0 &        0.732   &      1 &   1.414 &                   True \\
  & 2 & Min. trunc. error & 0.5 &           1 &    1.333 &   1.414 &                   True \\
  & 2 & SSP22 \cite{shu1988}          & 1   &          1 &        1 &   1.414 & True \\
  & 2 & SSP22* \cite{gottlieb2006} & 0.784   &      1.215   &  1.215 &   1.414 & True \\
\midrule
3 & 3 & Heun33 \cite{heun1900neue} & 0 &       0.776   &  1.333 &   1.817 &  False \\
  & 3 & SSP33 \cite{shu1988}   & 1 &            1 &        1 &   1.817 & True \\
\midrule
4 & 4 & RK44 (Kutta)          & 0 &        0.685 &        1 &   2.213 &                  False \\
  & 5 & Merson \cite{merson1957operational} & 0   &      0.242   &    0.5 &   3.309 & False \\
  & 10 &SSP104 \cite{ketcheson2008}       & 6 &            6 &        6 &   8.425 & False \\
\midrule
5 & 6 & Fehlberg \cite{fehlberg1969klassische}     & 0   & 0.057   &  0.125 & 3.727 & False \\
  & 7 & Dormand-Prince \cite{Dormand1980} & 0   &      0.040 &    0.086 &   4.789 & False \\
  & 7* & Bogacki-Shampine  \cite{Bogacki1996} & 0 &        0.313 &    0.859 &   5.827 & False \\
  & 7 & SSP75 \cite{ruuth2004} & 0 &        1.396 &    1.792 &   4.789 & False \\
  & 8 & SSP85 \cite{ruuth2004} & 0 &        1.875 &    1.919 &   5.827 & True \\
  & 9 & SSP95 \cite{ruuth2004} & 0 &        2.738 &    3.198 &   6.853 & False \\
\midrule
6 & 9 & Calvo \cite{calvo1990} & 0 &   0.021 &    0.059 &   6.265 & False \\
\midrule
8 & 13 & Prince-Dormand \cite{Prince1981} & 0   &      0.013   &  0.059 &   9.212 & False \\
\bottomrule
\end{tabular}
\caption{Properties of some Runge--Kutta methods and their optimal perturbations.  The optimal perturbed
radius of absolute monotonicity was computed by both the linear programming algorithm
and the iterated splitting algorithm; in every case they gave identical results (up to 
roundoff errors).  Decimal values have been truncated to the number of digits shown.
The Bogacki-Shampine method uses 8 stages but is first-same-as-last, so it is as efficient
as if it had 7 stages for non-rejected steps.
\label{tbl:explicit-methods1}}
\end{table}

 \begin{table}
\begin{tabular}{ll|l|ll|lll}
\toprule
Order & Stages & Method & $R(K)$ & $h_{[0,1]}$ & $\Rt(K)$ & $\tilde{h}_{[0,1]}$  &   \\
\midrule
1 & 1 & Forward Euler & 1 & 1.00 &          1 & 1.00 &       \\
\midrule
2 & 2 & Midpoint         & 0 & 0.02 &      0.732 & 0.73  &        \\
  & 2 & Min. trunc. error & 0.5 & 0.54 &          1 & 1.00 &      \\
  & 2 & SSP22 \cite{shu1988}          & 1 & 1.09 &          1 & 1.00 &   \\
  & 2 & SSP22* \cite{gottlieb2006} & 0.784 & 0.81 &      1.215 & 1.21 &      \\
\midrule
3 & 3 & Heun33 \cite{heun1900neue} & 0 & 0.00 &      0.776 & 0.90 &      \\
  & 3 & SSP33 \cite{shu1988}   & 1 & 1.00 &          1 & 1.00 &       \\
\midrule
4 & 4 & RK44 (Kutta)          & 0 & 0.17 &      0.685 & 0.68 &         \\
  & 5 & Merson \cite{merson1957operational} & 0 & 0.01 &      0.242 & 0.29    \\
  & 10 &SSP104 \cite{ketcheson2008}       & 6 & 6.04 &          6 & 6.00 &     \\
\midrule
5 & 6 & Fehlberg \cite{fehlberg1969klassische}     & 0 & 0.01 & 0.057 & 0.05 &     \\
  & 7 & Dormand-Prince \cite{Dormand1980} & 0 & 0.02 &      0.040 & 0.04 &      \\
  & 7* & Bogacki  \cite{Bogacki1996} & 0 & 0.06 &      0.313 & 0.31 &      \\
  & 7 & SSP75 \cite{ruuth2004} & 0 & 0.06 &      1.396 & 1.56 &      \\
  & 8 & SSP85 \cite{ruuth2004} & 0 & 0.08 &      1.875 & 1.87 &      \\
  & 9 & SSP95 \cite{ruuth2004} & 0 & 0.10 &      2.738 & 2.82 &      \\
\midrule
6 & 9 & Calvo \cite{calvo1990} & 0 & 0.04 & 0.021 & 0.02 &     \\
\midrule
8 & 13 & Prince-Dormand \cite{Prince1981} & 0 & 0.06 &      0.013 & 0.01 &      \\
\bottomrule
\end{tabular}
\caption{Problem \eqref{advect}: actual and predicted step size restrictions
for preservation of the numerical solution in the interval $[0,1]$; $R(K)$ and
$\Rt(K)$ are the SSP coefficient and the optimal perturbed SSP coefficient,
respectively, and $h_{[0,1]}$ and $\tilde{h}_{[0,1]}$   are the largest
step sizes that preserve the interval $[0,1]$ in practice for the unperturbed
method and the optimal perturbed method, respectively.}  
\label{tbl:explicit-methods}
\end{table}

We have computed optimal perturbations for several known explicit methods using the two 
algorithms described above.  In all cases, the two algorithms gave the same
values.  It thus seems possible that Algorithm \ref{alg:2}
also gives truly optimal results in general, but we do not have a proof.  Properties of the methods studied are given in Table \ref{tbl:explicit-methods1}.  
The values found have been truncated to three decimal places but are known to greater precision.  

For the 4-stage, order-four method of Kutta, the three-digit value of $\Rt(K)$ given in the table matches the value found by Shu and Osher.  However, the exact (irrational) value is slightly larger and is given in the Appendix.
The methods SSP75, SSP85, and SSP95 are optimal methods found in \cite{ruuth2004}, with property C  (see Remark \ref{propertyC}).  By considering  perturbed methods without property C, we obtain slightly larger coefficients for perturbations of SSP75 and SSP95.  On the other hand, relaxing the column assumption gives no benefit in the case of the SSP85 method. 

\subsubsection{Numerical test}
We also apply the methods -- both perturbed and unperturbed -- to the
variable-coefficient advection problem
\begin{align} \label{advect}
    u_t + (a(x,t) u)_x & = 0\, ,  \\
    u(0,t) & = 0\, ,  \nonumber \\
    u(x,0) & = g(x)\, ,  \nonumber \\
    a(x,t) & = \cos^4(200x + 400t)\, ,  \nonumber
\end{align}
representing a highly oscillatory flow field.  If $g(x)\in [0,1]$, then
the exact solution remains in $[0,1]$ for all time.    We consider the domain
$0<x<1$ and we semi-discretize using first-order upwind differencing in space 
on an equispaced grid with 20 points.  The forward invariance of
the interval $[0,1]$ is then preserved
by the explicit Euler method as long as $0 \le h \le 1$.  Application of any
Runge--Kutta method to this semi-discretization yields an iteration of the form
$$u_{n+1} = M(t_n,h) u_n$$
where $M(t_n, h)$ is a square matrix.  The initial vector $u_0$ is obtained from $g(x)$ at the spatial grid points; hence, if $g(x)\in[0,1]$, vector $u_0$ is also in $[0,1]$.  Consequently, the numerical solution of \eqref{advect} will
remain in $[0,1]$ if $g(x) \in [0,1]$ and $M$ satisfies
\begin{subequations} \label{Mcond}
\begin{align}
    m_{ij} & \ge 0 & \text{for all } i,j \\
    \sum_j m_{ij} & \le 1 & \text{for all } i.
\end{align}
\end{subequations}
In Table \ref{tbl:explicit-methods},    
in the column labeled $h_{[0,1]}$, we give the largest step size for which
the corresponding unperturbed method preserves the interval $[0,1]$; i.e.
the largest step size for which $M$ satisfies \eqref{Mcond}.
Similarly, in the column labeled $\tilde{h}_{[0,1]}$,
we give the largest step size for which the optimal perturbation of the method
(with first-order downwind differencing for the downwind operator)
preserves the interval $[0,1]$.  The values given are truncated (not rounded)
to two decimal places.  Most of the actual values agree very well with the
theoretical bounds.

Some additional interesting patterns are evident in the table and are discussed in the conclusions below.

\section{Conclusions}\label{sec:conclusions}
In this work we have studied SSP coefficients for perturbations of a given explicit Runge-Kutta method. We have considered both the linear and the nonlinear case, and have obtained useful bounds on the threshold factor and on the radius of absolute monotonicity
for perturbed Runge--Kutta methods.  
 We have also provided an algorithm for computing
optimal perturbations of explicit Runge-Kutta methods, and given optimal perturbations for many methods from the literature.
From Table \ref{tbl:explicit-methods} we see that
\begin{itemize}
	\item For most optimal SSP methods (up to order three), perturbation cannot yield a 
          larger coefficient.  This is evident already from the bound \eqref{Rpert}.
          For all other methods, some improvement is achieved.
	\item Consistent with Theorem \ref{Pbound}, for every method considered, 
    	  it is possible to achieve $\Rt>0$ by some perturbation.          
    \item The simple bound \eqref{Rpert} predicts the optimal coefficient to within a factor of three in
    		every case.
\end{itemize}

This work seems to provide a complete picture for the case of most interest: explicit
methods applied to nonlinear problems.
Nevertheless, some other interesting issues remain unsolved.  These include:
\begin{itemize}
	\item A method to compute optimal perturbations for linear problems.
	\item An algorithm for obtaining optimal splittings of implicit methods.
\end{itemize}
Besides, in this paper we have only considered perturbations $\tilde f$ such that  $\tilde h_0= h_0$ (see \eqref{FEcond1}-\eqref{FEcond2}), but the study done can be extended to the case $\tilde h_0\neq h_0$. In this way, a wider class of perturbations $\tilde f$ can be considered and larger SSP coefficients may be obtained. In a similar way, for fictitious perturbations (see Remark \ref{remark:ficticius}), monotonicity can be ensured with step size restrictions larger than the ones obtained with the results in this paper. 

These may be a starting point for future work.

\section{Proofs of the results in the paper}\label{sec:proofs}

This section contains the proofs of the different results in the paper (Theorems \ref{Theorem:1}-\ref{Pbound}, Propositions \ref{Proposition:1} and \ref{propfirstcol}, and Lemmas \ref{equivprop}-\ref{lemma_row}),  and an auxiliary lemma; the proofs of Propositions \ref{prop:rs1} and \ref{prop:rs2} and Theorem \ref{Theorem:5} are straightforward and they are omitted.

\begin{proof}  \textit{of \textbf{Theorem \ref{Theorem:1}.}}\\
The stability function $\phipert(z,\zt)$ is a bivariate polynomial and thus, for $r \le R(\phipert)$, 
it can be written in the form \eqref{dwtaylor1},
where the coefficients $\gamma_{j\ell}$ are non-negative and
(by consistency of the method) sum to unity.  Letting $z=hL$
and $\zt=h\tilde{L}$, applying $\|\cdot\|$, and using convexity
shows that  $\|\phipert(hL,-h\tilde{L})u\| \le \|u\|$.  \hfill $\square$
\end{proof}

\begin{proof}\textit{of \textbf{Proposition \ref{Proposition:1}.}}\\
From \eqref{bistabfuncDet},  the stability function \eqref{bistabfunc}  of any $s$-stage explicit perturbed Runge--Kutta method  is a bivariate polynomial of combined degree at most $s$. Furthermore, as $\phipert(z, -z)$ is the stability function of an $s$-stage Runge--Kutta scheme of linear order $p$, we have that $\phipert(z, -z)=\sum_{j=0}^p  z^j/j! + \sum_{j=p+1}^s \sigma_j z^j$. Thus, there is a bivariate polynomial $\Psi$ such that  $\phipert(z, \zt)=\phipert(z, -z)+(z+\zt)\Psi(z,\zt)$. As $\phipert$ has combined degree at most $s$, trivially  $\Psi$ is a polynomial of combined degree at most $s-1$.  \mbox{}\hfill $\square$
\end{proof}

\begin{lemma} \label{lem:1sp}
Let $\varphi(z)$ be a polynomial satisfying 
\begin{align}\label{polpert} 
\varphi(z) & = 1 + \gamma_1 z +\cdots + \gamma_p z^p +\gamma_{p+1} z^{p+1} + \cdots + \gamma_s z^s\\
\gamma_j & \geq \frac{1}{j!}, \quad  j=1, \ldots, p\,  . \nonumber
\end{align}
Then the radius of absolute monotonicity of  $\varphi$ satisfies
\begin{align}R(\varphi) \le\sqrt[p]{s(s-1)\ldots (s-p+1)}\, . \label{R1sp}
\end{align}
\end{lemma}

\begin{proof}  \textit{of \textbf{Lemma \ref{lem:1sp}.}}\\
If  $R(\varphi)=0$, inequality \eqref{R1sp}
is trivial.  Let $\varphi(z)$ satisfy \eqref{polpert}
and be absolutely monotonic at $-r$ with $r>0$.
Then it can be written as 
\begin{align}\label{Taylorphi}
\varphi(z)=\sum_{j=0}^s \alpha_j \left(1 + \frac{z}{r}\right)^j=\sum_{j=0}^s \alpha_j \left(\sum_{\ell=0}^j  \frac{z^\ell}{r^\ell} {j \choose \ell}\right)=\sum_{\ell=0}^s \left(\sum_{j=\ell}^s \alpha_j  {j \choose \ell}\right ) \frac{z^\ell}{r^\ell}\,, 
\end{align}
where $\alpha_j\geq 0$.  Observe that from \eqref{polpert} we get $\varphi(0)=1$, and thus in \eqref{Taylorphi} we have $\sum_{j} \alpha_j=1$. As $\varphi$ is of the form \eqref{polpert},   the coefficient of $z^p$ is larger than $1/p!$. Some computations give 
$$\frac{1}{p !} \leq \left(\sum_{j=p}^s \alpha_j  {j \choose p}\right ) \frac{1}{r^p} \leq  \left(\sum_{j=p}^s \alpha_j  \right ) {s \choose p} \, \frac{1}{r^p} \leq  {s \choose p} \frac{1}{r^p}=\frac{s \, (s-1) \cdots (s-p+1)}{ p ! \, \, r^p}. 
$$
Consequently, 
$ r\leq \sqrt[p]{s(s-1)\cdots (s-p+1)}\, .  $ 
\mbox{}\hfill $\square$
\end{proof}

 We remark that equality in \eqref{R1sp} is obtained for the polynomial
\begin{equation}\varphi(z)=\left(1+ \frac{z}{r}\right)^s\, , \label{optTihi}
\end{equation}
where $r=\sqrt[p]{s(s-1)\cdots (s-p+1)}$. 

\begin{proof}\textit{of \textbf{Theorem \ref{Theorem:2}.}}\\
If $R(\psi)=0$ for all $\psi \in \Pit_{s,p}$, then $\cRt_{s,p}=0$ and
inequality \eqref{incRtsp} is true. Otherwise, there exists a function  $\psi
\in \Pit_{s,p}$  a.m. at $(-r,-r)$ with $r>0$.  
By \cite[Lemmas 2.9 and 2.10]{higueras2006}, $\psi$ is a.m.
at the points $(\xi, \xi)$, with $\xi \in [-r, 0]$.  
Writing $\psi(z, \zt) = \sum \sum \mu_{jk} z^j \zt^k$
and differentiating shows that all coefficients $\mu_{jk}$ are non-negative
since $\psi$ is a.m. at $(0,0)$.  Thus  $\psi(z,z)$ (viewed as a function
of one variable) is of the form \eqref{polpert} and is a.m. at $-r$.
Application of Lemma \ref{lem:1sp} gives the desired result. \hfill $\square$
\end{proof}

\begin{proof}\textit{of \textbf{Theorem \ref{thm:feasible}}.}\\
From \cite[Prop. 3.7]{higueras2005a}, we have
$R(\bbA,\Apert)>0$ if and only if the Butcher coefficients satisfy
\begin{align}\label{ineqa}
    \bbA + \Apert \geq 0 \, , \qquad \Apert \geq 0\, ,
\end{align}
and the following inequalities  hold,
\begin{subequations} \label{ineqb}
\begin{eqnarray}
&& \hbox{Inc }((\bbA+ 2 \Apert)  (\bbA+ \Apert)) \leq 
\hbox{Inc }(\bbA+ \Apert)\, , \label{3.6} \\[1ex]
&& \hbox{Inc }((\bbA+ 2 \Apert) \Apert ) \leq 
\hbox{Inc }(\Apert)\, , 
\end{eqnarray}
\end{subequations}
where $\hbox{Inc } (F)$ denotes the incidence matrix of matrix $F$ defined as
$\hbox{Inc } (F)=(g_{ij})$ where $g_{ij}=1$ if $f_{ij}\neq 0$, and $g_{ij}=0$ 
if $f_{ij}=0$. 

Consider first the implicit case.  By making all entries of $\Apert$ positive
we can satisfy \eqref{ineqb}, and by making them large enough we can satisfy
\eqref{ineqa}.  For the explicit and diagonally implicit cases, note that if $\bbA, \Apert$ are (strictly) lower-triangular, then the left-hand sides
of \eqref{ineqb} are also.  Thus by making all the (strictly) lower-triangular
entries of $\Apert$ positive, and by taking them large enough, we can 
satisfy the above inequalities. \hfill $\square$
\end{proof}

\begin{proof}\textit{of \textbf{Theorem \ref{prop_vpos}.}}\\
Let $r=\Rt(\bbA)$.  Then $\gamma_r=(I-\alphaup-\alphadown)e\geq 0$, and thus
from \eqref{canonical-coefficients}  and \eqref{zerowell}   we get
\begin{equation}(I-2\alphadown_r)v_r   \geq 0\, . \label{vpos}
\end{equation}
As $\alphadown_r\geq 0$, and since we consider only explicit, zero-well-defined perturbations,
$I-2\alphadown_r$ is an $M$ matrix.  Thus $(I-2\alphadown_r)^{-1}\geq 0$. If we multiply \eqref{vpos} by  $(I-2\alphadown_r)^{-1}$ we obtain that $v_r\geq 0$. \mbox{} \hfill $\square$
\end{proof}

\begin{proof}\textit{of \textbf{Theorem \ref{Pbound}.}}\\
The proof is similar to that of \cite[Lemma 3.2]{ruuth2004}.
Consider an optimal perturbation $\Apert$ and set $r=\Rt (\bbA, \Apert)>0$;  consider too the canonical representation
\eqref{dwRunge--Kutta_cso}. Let
$\Lambda=\alphaup_r+\alphadown_r=(\alpha_{ij})$, $\Gamma=
\alphaup_r/r=(\beta_{ij})$, $\tilde \Gamma= \alphadown_r/r=(\tilde\beta_{ij})$;
observe that $\Lambda, \Gamma, \tilde \Gamma\geq 0$, and that $\Lambda=r
(\Gamma+\tilde \Gamma)$.  As $(I-\Lambda)e=\gamma_r\geq 0$ and $\alpha_{ik}\geq
0$, we have $\alpha_{ik}\leq 1$; as $(I-\Lambda)\bbA=\Gamma-\tilde \Gamma$, we
have
\begin{align}a_{ik}= \beta_{ik}-\tilde \beta_{ik} + \sum_{j=k+1}^{i-1} \alpha_{ij} a_{jk}\, . \label{Acoef}
\end{align}
As $\alpha_{ik}= r (\beta_{ik} + \tilde \beta_{ik})$, then $\beta_{ik} + \tilde \beta_{ik}= \alpha_{ik}/r\leq 1/r$. In particular, from \eqref{Acoef}, 
$$|a_{21}|=\left|\beta_{21} - \tilde \beta_{21}\right|\leq \beta_{21} + \tilde \beta_{21}\leq \frac{1}{r}\, . $$
We proceed by induction on row $\ell$ of $\bbA$. Assume that $|a_{ij}|\leq 1/r$, for $i= 2, \ldots, \ell$,  $j=1, \ldots , \ell-1$, and consider row $\ell+1$. Then, from  \eqref{Acoef}, 
\begin{align*}|a_{\ell+1, 1}|&=\left|\beta_{\ell+1, 1}-\tilde \beta_{\ell+1, 1} + \sum_{j=2}^{\ell} \alpha_{\ell+1, j} \, a_{j, 1}\right|\leq \beta_{\ell+1, 1}+\tilde \beta_{\ell+1, 1} + \sum_{j=2}^{\ell} \alpha_{\ell+1, j} |\, a_{j, 1}|\\
&\leq  \frac{1}{r} \alpha_{\ell+1, 1}+ \frac{1}{r} \sum_{j=2}^{\ell} \alpha_{\ell+1, j} \leq  \frac{1}{r} \sum_{j=1}^{\ell} \alpha_{\ell+1, j}  \leq \frac{1}{r}  \, . 
\end{align*}
A similar argument can be used to show that $|a_{\ell+1,j}|\leq 1/r$, $j= 2, \ldots, \ell$. The Theorem follows by induction. \hfill $\square$
\end{proof}

\begin{proof}\textit{of \textbf{Proposition \ref{propfirstcol}.}}\\
It is easily seen that the modified method is equivalent to the original
one when $f=\ft$, so they correspond to the same unperturbed method.
Meanwhile, the transformation never leads to negative coefficients,
so the modified method is a.m. at $r$. \hfill $\square$
\end{proof}

\begin{proof}\textit{of \textbf{Lemma \ref{equivprop}.}}\\
To prove the first part,
take $\ft = f$ in \eqref{dwRunge--Kutta_cso} to obtain:
\begin{align*}
Y & = \gamma_r u_n + (\alphaup_r + \alphadown_r) Y
                    + (\alphaup_r - \alphadown_r) \frac{\Dt}{r}F.
\end{align*}
Subtract $2\alphadown_r Y$ from both sides to get
\begin{align} \label{zwd}
(I-2\alphadown_r) Y & = \gamma_r u_n + (\alphaup_r - \alphadown_r) 
                    \left(Y +  \frac{\Dt}{r}F\right).
\end{align}
Substituting \eqref{cso} in the above gives
\begin{align*}
(I-2\alphadown_r) v_r u_n + (I-2\alphadown_r)\alpha_r \left(Y +  \frac{\Dt}{r}F\right)
    & = \gamma_r u_n + (\alphaup_r - \alphadown_r) \left(Y +  \frac{\Dt}{r}F\right),
\end{align*}
Equating coefficients yields \eqref{alpharel1}.

To prove the second part, assume $I-2\alphadown_r$ is invertible and
write \eqref{alpharel1} as
\begin{subequations} \label{alpharel2}
\begin{align}
 \alpha_r & = (I-2\alphadown_r)^{-1}(\alphaup_r - \alphadown_r) \\
 v_r & = (I-2\alphadown_r)^{-1}\gamma_r.
\end{align}
\end{subequations}
Substitute \eqref{alpharel2} in \eqref{cso}, multiply on the left by
$(I-2\alphadown)^{-1}$, and follow the steps above in reverse. \mbox{}\hfill $\square$
\end{proof}

\begin{proof}\textit{of \textbf{Lemma \ref{signlemma}.}}\\
Since the perturbation is zero-well-defined, we can define
\begin{align}\label{a1a2}
    \alpha^+   = (I-2\alphadown)^{-1} \alphaup_r\, ,  \qquad \qquad 
    \alpha^-   = (I-2\alphadown)^{-1} \alphadown_r\, .
\end{align}
Then, by \eqref{alpharel1a}, $\alpha_r=\alpha^+-\alpha^-$.  Furthermore, since $I-2\alphadown_r$ is an $M$-matrix,
we have $\alpha^+\ge0$ and $\alpha^-\ge0$.
Solving \eqref{a1a2} for $\alphaup_r,\alphadown_r$ gives \eqref{alphar}. \hfill $\square$
\end{proof}

 \begin{proof}\textit{of \textbf{Lemma \ref{lemma_row}.}}\\
 If $j_0$ is the first row with negative terms in $\alphaup$ or $\alphadown$, straightforward computations give that $\hat \alpha^\textup{up}_{i,j}=\alphaup_{i,j}$ and $\hat \alpha^\textup{down}_{i,j}=0$ for $1\leq i\leq j_0-1$, $1\leq j\leq j_0-2$, and
\begin{subequations} \begin{align}
\hat \alpha^\textup{up}_{j_0,1}&=  \alpha^\textup{up}_{j_0,1} -2  \sum_{i=2}^{j_0-1} \left((\alpha^\textup{up}_{j_0,i})^-+ (\alpha^\textup{down}_{j_0,i})^-\right) \,  \alpha^\textup{up}_{i,1}\, , \label{d1} \\\quad
\hat \alpha^\textup{up}_{j_0,\ell}&=  (\alpha^\textup{up}_{j_0,\ell})^++ (\alpha^\textup{down}_{j_0,\ell})^- -2  \sum_{i=\ell+1}^{j_0-1} \left((\alpha^\textup{up}_{j_0,i})^-+ (\alpha^\textup{down}_{j_0,i})^-\right) \, \alpha^\textup{up}_{i,\ell} \, , \quad \ell=2, \ldots, j_0-1\, .\label{dm} 
\end{align}
\end{subequations}
and
\begin{subequations}
\begin{align}
\hat \alpha^\textup{down}_{j_0,1}&=  \alpha^\textup{down}_{j_0,1} -2  \sum_{i=2}^{j_0-1} \left((\alpha^\textup{up}_{j_0,i})^-+ (\alpha^\textup{down}_{j_0,i})^-\right) \, \alpha^\textup{down}_{i,1}\, , \label{d11} \\\quad
\hat \alpha^\textup{down}_{j_0,\ell}&=  (\alpha^\textup{up}_{j_0,\ell})^-+ (\alpha^\textup{down}_{j_0,\ell})^+  -2  \sum_{i=\ell+1}^{j_0-1} \left( (\alpha^\textup{up}_{j_0,i})^-+ (\alpha^\textup{down}_{j_0,i})^-\right) \, \alpha^\textup{down}_{i,\ell} \, , \quad \ell=2, \ldots, j_0-1\, .\label{dm1} 
\end{align}
\end{subequations}
Let $m_0$ be the largest index $m_0\geq 1$ such that  $\alphaup_{j_0,m_0}<0$ or $\alphadown_{j_0,m_0} <0$. In this case,
$\alphaup_{j_0,i}\geq 0$, $\alphadown_{j_0, i}\geq 0$ for $i=m_0+1, \ldots, j_0-1$, and  thus 
$$(\alpha^\textup{up}_{j_0,i})^+=\alpha^\textup{up}_{j_0,i}\, , \quad (\alpha^\textup{down}_{j_0,i})^+=\alpha^\textup{down}_{j_0,i}\, , \quad (\alpha^\textup{up}_{j_0,i})^-=(\alpha^\textup{down}_{j_0,i})^-=0\, ,\qquad i=m_0+1, \ldots, j_0-1\, . $$ 
 If $m_0=1$, from \eqref{d1} and \eqref{d11} we get  $\hat \alpha^\textup{up}_{j_0,1} = \alpha^\textup{up}_{j_0,1}$ and $\hat \alpha^\textup{down}_{j_0,1}= \alpha^\textup{down}_{j_0,1}$,  and thus $\hat \alpha^\textup{up}_{j_0,1}<0$ or $\hat \alpha^\textup{down}_{j_0,1}<0$.  If $m_0\ge  2$, from \eqref{dm} and \eqref{dm1} we get
 $$\hat \alpha^\textup{up}_{j_0,m_0} = (\alpha^\textup{up}_{j_0,m_0})^++ (\alpha^\textup{down}_{j_0,m_0})^- \geq 0\, , \qquad \hat \alpha^\textup{down}_{j_0,m_0} =  (\alpha^\textup{up}_{j_0,m_0})^-+(\alpha^\textup{down}_{j_0,m_0})^+ \geq 0\, . 
 $$
Finally, for $1\leq m_0\leq j_0-2$, from  \eqref{dm} and \eqref{dm1} we get that, for $\ell=m_0+1, \ldots, j_0-1$, we have
$$\hat \alpha^\textup{up}_{j_0,\ell} = (\alpha^\textup{up}_{j_0,\ell})^+ \geq 0\, ,  \qquad \hat \alpha^\textup{down}_{j_0,\ell} =  (\alpha^\textup{down}_{j_0,\ell})^+ \geq 0\,.   
$$
\mbox{}\hfill $\square$
 \end{proof}

\section{Appendix}

In this section we give additional details on SSP coefficients and optimal perturbations of
second order 2-stage Runge--Kutta methods and the classical 4-stage fourth order Runge--Kutta method. 

\subsection{Second order 2-stage methods} \label{sec:2stage-details}

We consider the family of 2-stage second order methods \eqref{2eo2OptLin}. In example \ref{ex:2eo2}   we studied perturbations that increase the SSP coefficient for the linear case. For nonlinear problems, in example \ref{ex_2stage2order},  figure \ref{f22} shows the values of $\Rt(\bbA)$ for $\alpha\in[-3,3]$. 

In this section,  for each $\alpha$,  we give the expressions for $\Rt(\bbA)$ and we show optimal perturbations $\Apert_{NL}$ such that $R(\bbA, \Apert_{NL})=\Rt(\bbA)$.  It is important to point out the convenience of choosing 
$\Apert_{NL}= \Apert_{L}$, where  $\Apert_{L}$ denotes the optimal perturbation for the linear case. In this case, we have not only $R(\bbA, \Apert_{L})=\Rt(\bbA)$ but also $\RLin(\bbA,  \Apert_{NL})=\RLinopt(\bbA)$. The computations required to obtain the results in this section have been done with the symbolic computation program \emph{Mathematica}.

If we denote by $r=\Rt(\bbA)$,  we have that
\begin{equation}
r= \begin{cases}
\displaystyle \frac{1}{|\alpha|}\, ,  &  \qquad \displaystyle  \hbox{if} \quad  \alpha \in\left(-\infty, -\frac{1}{2} \left(1+\sqrt{7}\right)\right]\bigcup  \left[\frac{1}{2} \left(-1+\sqrt{7}\right), \infty\right)\,,  \\[3ex]
\displaystyle   \frac{-1+\alpha+\sqrt{ 3 \alpha ^2-2 \alpha +1} 
   }{|\alpha| }\, ,  &   \qquad \displaystyle  \hbox{if} \quad  \alpha\in\left(-\frac{1}{2} \left(1+\sqrt{7}\right), 0\right) \bigcup \left(0, \frac{1}{2} \left(-1+\sqrt{7}\right)\right)\, . \end{cases}\label{radiopert22}
\end{equation}
Next we give optimal perturbations $\Apert_{NL}$. 

For $\alpha<0$, we obtain that it is not possible to obtain a perturbation of the form \eqref{2eo2OptLin}  with $\tilde b_2=0$ and $\tilde a_{21}=0$. Consequently, $\Apert_{NL}\neq \Apert_{L}$ and we always have that $\RLin(\bbA,  \Apert_{NL})<\RLinopt(\bbA)$. Optimal perturbations of the form \eqref{2eo2OptLin} for different values of $\alpha<0$ must satisfy the following conditions.
\begin{itemize}   
\item For $-\frac{1}{2} \left( 1+\sqrt{7}\right)\leq\alpha <0$, the coefficients $\tilde a_{21}$, $\tilde b_1$ and $\tilde b_2$ in $\Apert_{NL}$ must satisfy
   $$-\alpha \leq \tilde a_{21}\leq \frac{1-r \, \alpha}{2 \, r}\, , \qquad \tilde b_1=-\frac{r\, \tilde a_{21}}{2\, \alpha } \, , \qquad \tilde b_2=-\frac{1}{2\, \alpha}\, , 
   $$
   where $r=\Rt(\bbA)$.
\item For $\alpha\leq-\frac{1}{2} \left( 1+\sqrt{7}\right)$,    
   we should have 
   $$\tilde a_{21}=-\alpha,\qquad  -\frac{1}{2 \,\alpha }\leq\tilde b_1\leq \frac{-2 \,\alpha ^2-2
 \,  \alpha +1}{4\, \alpha }\, , \qquad  -\frac{1}{2 \,\alpha }\leq
  \tilde b_2\leq \frac{2\, \alpha  \,\tilde b_1-1}{4\, \alpha
   }\, . $$
\end{itemize}
For  $\alpha>0$ we can find optimal perturbations with $\tilde b_2=0$ and $\tilde a_{21}=0$. Coefficient $\bt_1$ must satisfy the following conditions.
\begin{itemize}
\item For $0<\alpha \leq  
   \left(-1+\sqrt{7}\right)/2$,  we have that   
   \begin{align}  \qquad \bt_1=\frac{\sqrt{3 \alpha ^2-2
   \alpha +1}-\alpha}{2 \alpha }  \, . \label{tb1NL}
   \end{align}   
Thus there is a unique $\Apert_{NL}$ of the form \eqref{2eo2OptLin}. In this case, we   have  $R(\bbA)<R(\bbA, \Apert_{NL})=\Rt(\bbA)$.

   \item For $\left(-1+\sqrt{7}\right)/2<\alpha< 1$, we also get $R(\bbA)<\Rt(\bbA)$,  but in this case the optimal perturbation  $\Apert_{NL}$ is not unique. All the perturbations with $\bt_1$ satisfying
$$   \frac{1-\alpha }{\alpha } \leq \bt_1\leq
   \frac{2 \alpha ^2-2 \alpha +1}{4 \alpha } \, , 
$$
are optimal. In particular, we can take $\Apert_{NL}=\Apert_{L}$. With this choice, $R(\bbA, \Apert_{L})=\Rt(\bbA)=1/\alpha$ and $\RLin(\bbA,  \Apert_{L})=\RLinopt(\bbA)\approx 1.22$. Furthermore,   $\alpha=\left(-1+\sqrt{7}\right)/2$  provides the largest  SSP coefficient within the family of 2-stage second order method (see figure \ref{f22}).

\item For $1\leq \alpha$, we have $R(\bbA)=\Rt(\bbA)=1/\alpha$ and the optimal perturbation $\Apert_{NL}$  is not unique. All the values
$$0\leq \bt_1\leq
   \frac{2 \alpha ^2-2 \alpha +1}{4 \alpha } 
$$
give  optimal perturbations. We can take $\Apert_{NL}=0$,   but in this case $\RLin(\bbA, 0)< \RLinopt(\bbA)$. A better choice is $\Apert_{NL}=\Apert_{L}$. Observe that, for $\alpha=1$, we get the optimal SSP coefficient $R(\bbA)=1$ that cannot be increased by perturbations.
  \end{itemize}

Next, we consider some concrete values of $\alpha$ to show the the expressions of the perturbations. For each value, we give the Butcher tableau of the perturbation and matrices  $\alphaup$ and $\alphadown$ in \eqref{dwRunge--Kutta_cso}. 
\begin{itemize}
\item For $\alpha=1/2$ we get method RK2a in \cite{HuKoVa} with $R(\bbA)=0$. With perturbation
\begin{align} \label{midpointRKpert}
\Apert= \begin{pmatrix}
0 & 0 & 0 \\
0 & 0 & 0\\
\bt_1 & 0 & 0
\end{pmatrix} , \, \alphaup= \begin{pmatrix}
0 & 0 & 0 \\
\bt_1 & 0 & 0\\
0 & 2 \bt_1  & 0
\end{pmatrix}
  , \,  \alphadown = \begin{pmatrix}
0 & 0 & 0 \\
0 & 0 & 0\\
1-2\bt_1  & 0 & 0
\end{pmatrix}
 , \, \gamma = \begin{pmatrix}
1\\
1-\bt_1\\
0
\end{pmatrix}, 
\end{align}
where $\bt_1=\frac{1}{2} \left(\sqrt{3}-1\right)$, we get  $R(\bbA,\Apert) =\RLin(\bbA, \Apert)= \sqrt{3}-1$. 
\item For $\alpha=2/3$,  we have a nontrivial SSP coefficient $\Rt(\bbA)=1/2$, but we can increase this value to $R(\bbA,\Apert_1)=\Rt(\bbA)=1$ with perturbation
\begin{align*}
\Apert_1= \begin{pmatrix}
0 & 0 & 0 \\
0 & 0 & 0\\
1/4 & 0 & 0
\end{pmatrix}
\, , \, \alphaup= \begin{pmatrix}
0 & 0 & 0 \\
2/3 & 0 & 0\\
0 & 3/4 & 0
\end{pmatrix}
\, , \,  \alphadown = \begin{pmatrix}
0 & 0 & 0 \\
0 & 0 & 0\\
1/4 & 0 & 0
\end{pmatrix}
\, , \, \gamma  = \begin{pmatrix}
1\\
1/3\\
0
\end{pmatrix}\, . 
\end{align*}
For this perturbation, $R(\phi_\bbA)=\RLin(\bbA, \Apert_1)=1$. We can take $\gamma=(1,0,0)^t$ by modifying the first column of $\alphaup$ and $\alphadown$ according to \eqref{transform},
\begin{align*}
\Apert_2= \begin{pmatrix}
0 & 0 & 0 \\
1/6 & 0 & 0\\
3/8 & 0 & 0
\end{pmatrix}\, , \, \alphaup= \begin{pmatrix}
0 & 0 & 0 \\
5/6 & 0 & 0\\
0 & 3/4 & 0
\end{pmatrix}
\, , \,  \alphadown = \begin{pmatrix}
0 & 0 & 0 \\
1/6 & 0 & 0\\
1/4 & 0 & 0
\end{pmatrix}
\, , \, \gamma = \begin{pmatrix}
1\\
0\\
0
\end{pmatrix}. 
\end{align*}

\item As it has been pointed out above, the largest value in the $\alpha$-family of 2-stage second order schemes is $\Rt(\bbA)=(1+\sqrt{7})/3$ and it is obtained for $\alpha=(\sqrt{7}-1)/2$. The perturbation is  of the form \eqref{2eo2OptLin} with $\bt_1= \left(\sqrt{7}-2\right)/2$, and
\begin{align*}
\alphaup= \begin{pmatrix}
0 & 0 & 0 \\
1  & 0 & 0\\
0 & \frac{1}{9} \left(4+\sqrt{7}\right)  & 0
\end{pmatrix}
\, , \,  \alphadown = \begin{pmatrix}
0 & 0 & 0 \\
0 & 0 & 0\\
\frac{1}{9} \left(5-\sqrt{7}\right)  & 0 & 0
\end{pmatrix}
\, , \, \gamma_r = \begin{pmatrix}
1\\
0 \\
0
\end{pmatrix}
\end{align*}
This is the perturbation obtained in \cite[Table V]{gottlieb2006} by numerical search in the class of perturbations considered in \cite{gottlieb2006}. 
\end{itemize}

\subsection{Classical fourth order 4-stage method}\label{4eo4L}

For nonlinear problems, applying the analysis above, we find that the optimal 
perturbation of the classical method has SSP coefficient given by
the real root of $x^3+2 x^2+4 x-4=0$, which is approximately
$\Rt(\bbA)\approx 0.685016$.  The corresponding perturbation is
{\em not unique.}  For instance, we can take $\gamma_r=(1,0,0,0,0)$,
and all entries of $\alphadown_r$ equal to zero except
\begin{align}
	(\alphadown_r)_{31} & = \frac{r^2}{4}   & (\alphadown_r)_{42} & = \frac{r^2}{2},  
\end{align}
where $r=\Rt(\bbA)$.
However, there exist other optimal perturbations with additionally $(\alphadown_r)_{42}=\epsilon$ where $0\le \epsilon \le 0.782$.  

We remark that nearly-optimal perturbations for this method are given
in \cite[p. 448]{shu1988} and \cite{Higueras2010}.
Interestingly, these different perturbed methods have different values of
$\RLin(\bbA, \Apert)$.


\end{document}